\newcommand{\tle}{
The maximal mean equicontinuous factor via regional mean sensitivity
}

\documentclass[12pt, reqno]{amsart}
\usepackage{amsmath, amsthm, amscd, amsfonts, amssymb, graphicx, xcolor}
\usepackage[bookmarksnumbered, colorlinks, plainpages]{hyperref}
\usepackage{todonotes}
\usepackage{multicol}

\newtheorem{theorem}{Theorem}[section]
\newtheorem*{theorem*}{Theorem}
\newtheorem{lemma}[theorem]{Lemma}
\newtheorem{proposition}[theorem]{Proposition}
\newtheorem{corollary}[theorem]{Corollary}
\newtheorem*{corollary*}{Corollary}
\theoremstyle{definition}
\newtheorem{definition}[theorem]{Definition}
\newtheorem{example}[theorem]{Example}
\newtheorem*{example*}{Example}

\theoremstyle{remark}
\newtheorem{remark}[theorem]{Remark}
\numberwithin{equation}{section}

\newcommand{\dd}[1]{ \, \mathrm{d}#1}
\newcommand{\supp}{\operatorname{supp}}
\newcommand{\haar}[1]{\left|#1\right|}
\newcommand{\seq}[1]{\vec{\textbf{#1}}}
\newcommand{\icerhull}[1]{\langle #1 \rangle}
\newcommand{\upperAdens}{\operatorname{ua-dens}}

\newcommand{\upperBdens}{\operatorname{ub-dens}}

\def\P{\mathrm{P}}
\def\RP{\mathrm{Q}_\mathrm{rp}}
\def\RME{\mathrm{R}_\mathrm{me}}
\def\RMS{\mathrm{Q}_\mathrm{rms}}
\def\SM{\mathrm{S}_\mathrm{rjms}}
\def\WSM{\mathrm{S}_\mathrm{wsm}}

\newcommand{\upperElement}[1]{{#1}^{\blacktriangle}}
\newcommand{\lowerElement}[1]{{#1}^{\blacktriangledown}}

\begin{document}
\setcounter{page}{1}

\color{black}{
\noindent 
}


\title[\tle]{\tle}

\author[]{Till Hauser}
\address[T. Hauser]{Facultad de Matem\'aticas, Pontificia Universidad Cat\'olica de Chile. Edificio Rolando Chuaqui, Campus San Joaquín. Avda. Vicuña Mackenna 4860, Macul, Chile.}
\email{hauser.math@mail.de}
\author[]{Chunlin Liu }
\address[C. Liu]{School of Mathematical Sciences, Dalian University of Technology, Dalian, 116024, P.R. China
and 
Institute of Mathematics, Polish Academy of Sciences, ul. Śniadeckich 8, 00-656 Warszawa, Poland
}
\email{chunlinliu@mail.ustc.edu.cn}
\thanks{This article was funded by the Deutsche Forschungsgemeinschaft (DFG, German Research Foundation) – 530703788. 
The second author was supported by the
Postdoctoral Fellowship Program and China Postdoctoral Science Foundation under Grant Number BX20250067, and the China Postdoctoral Science Foundation under Grant Number 2025M773074.
We are grateful to Mar\'{\i}a Isabel Cortez and Wen Huang for valuable suggestions on the presentation of the article. 
}
\begin{abstract}
For actions of amenable groups, mean equicontinuity—a natural relaxation of equicontinuity obtained by averaging metrics along orbits—is well known to yield a maximal mean equicontinuous factor.  
In 2021, Li and Yu introduced the notion of weak sensitivity in the mean for actions of $\mathbb{Z}$ to gain a deeper understanding of this phenomenon, building on earlier work by Qiu and Zhao.  

We demonstrate that this relation is insufficient for actions of non-Abelian groups. To overcome this limitation, we introduce the regional mean sensitive relation, which more precisely captures the dynamical behaviour underlying the maximal mean equicontinuous factor. We discuss its fundamental properties and highlight its advantages in the non-Abelian setting.  
In particular, we show that mean equicontinuity is equivalent to the nonexistence of non-diagonal regional mean sensitive pairs.  
For this, we work in the context of actions of $\sigma$-compact and locally compact amenable groups.

\noindent \textit{Keywords.} Mean equicontinuity, maximal mean equicontinuous factor, regionally mean sensitive pairs, amenable group actions, maximal support.
\newline
\noindent \textit{2020 Mathematics Subject Classification.} 
Primary 
37B05;  
Secondary  
37A15, 
37B25. 
\end{abstract} 
\maketitle






\section{Introduction}
The classical notion of equicontinuity plays a central role in topological dynamics, ensuring the existence of a maximal equicontinuous factor for any group action. 
In the context of actions of amenable groups, a natural relaxation of equicontinuity is given by the notion of $(\mathcal{F}-)$mean equicontinuity. It is motivated by the study of discrete spectrum \cite{fomin1951dynamical}, 
the interplay of topological dynamics and ergodic theory \cite{fuhrmann2022structure}, as well as the study of various prominent examples, such as regular Toeplitz actions,  Sturmian subshifts and Auslander systems. For details see \cite{downarowicz2005survey, cortez2014invariant,fuhrmann2022structure} and the references within. 

To describe this concept in more detail, let $G$ be a $\sigma$-compact and locally compact amenable group and denote $\haar{\cdot}$ for a choice of a Haar measure on $G$.
A sequence $\mathcal{F}=(F_n)_{n\in \mathbb{N}}$ of compact subsets of $G$ with $\haar{F_n}>0$ is called \emph{(left) F{\o}lner} if for each non-empty and compact subset $K\subseteq G$ we have 
$\haar{KF_n\Delta F_n}/\haar{F_n}\to 0$, where $\Delta$ denotes the symmetric difference.
For a F{\o}lner sequence $\mathcal{F}=(F_n)_{n\in \mathbb{N}}$ and a continuous metric $d$ on $X$ we define the \emph{(Besicovitch) $\mathcal{F}$-mean pseudometric} by 
\[D_\mathcal{F}(x,x'):=\limsup_{n\to \infty}\frac{1}{\haar{F_n}}\int_{F_n}d(g.x,g.x') \dd g.\]
An action is called \emph{$\mathcal{F}$-mean equicontinuous} if $D_\mathcal{F}\in C(X^2)$. It is well-known that this notion is independent of the choice of a continuous metric on $X$ \cite{fuhrmann2022structure}. 
Whenever $G$ is Abelian (or whenever $(X,G)$ is minimal), then this notion does not depend on the choice of a F{\o}lner sequence $\mathcal{F}$ \cite{fuhrmann2022structure}. 
However, there exist transitive actions of countable amenable groups where this notion depends on the choice of a F{\o}lner sequence $\mathcal{F}$ \cite{fuhrmann2025continuity}. 
We thus keep the choice of a F{\o}lner sequence explicit. 

Clearly, the trivial action on one point is $\mathcal{F}$-mean equicontinuous and the restriction of a $\mathcal{F}$-mean equicontinuous action to a non-empty closed invariant subset is also $\mathcal{F}$-mean equicontinuous. Furthermore, whenever $(X_{n},G)_{n\in \mathbb{N}}$ is a countable family of $\mathcal{F}$-mean equicontinuous actions, then it is easy to verify that the product action 
$(\prod_{n=1}^\infty X_{n},G)$
is also $\mathcal{F}$-mean equicontinuous (see e.g. \cite{li2015mean} for $\mathbb{Z}$-actions). 

It thus follows from a well-known argument \cite{auslander1988minimal, hauser2025mean} that for any action $(X,G)$ there exists a maximal $\mathcal{F}$-mean equicontinuous factor $(X_{\text{$\mathcal{F}$-me}},G)$, which is uniquely determined up to conjugacy. 
We denote $\pi_{\text{$\mathcal{F}$-me}}\colon X\to X_{\text{$\mathcal{F}$-me}}$ for the respective factor map and 
\[\RME^\mathcal{F}(X):=\{(x,x')\in X^2;\, \pi_{\text{$\mathcal{F}$-me}}(x)=\pi_{\text{$\mathcal{F}$-me}}(x')\}\]
for the \emph{$\mathcal{F}$-mean equicontinuous structure relation}. 

Obviously, such an existential proof cannot tell us much about the maximal $\mathcal{F}$-mean equicontinuous factor and the $\mathcal{F}$-mean equicontinuous structure relation $\RME^\mathcal{F}(X)$ and it is natural to ask whether there exists a relation $Q$ on $X$ that allows for a better understanding, similar to the regionally proximal relation in the context of the maximal equicontinuous factor \cite[Chapter 9]{auslander1988minimal}. 

For actions of $\mathbb{Z}$ this question was resolved in \cite{QiuZhao2020}. 
Building\footnote{
For more details see the remark following \cite[Theorem~5.11]{li2021mean}.
} 
on their work in \cite{li2021mean} the respective ideas were further developed and led to the notion of \emph{($\mathcal{F}$-)weak sensitivity in the mean} in the context of actions of $\mathbb{Z}$ and the standard F{\o}lner sequence $\mathcal{F}=(F_n)_{n\in \mathbb{N}}$ given by $F_n:=\{0,\dots, n-1\}$.
To give the definition of the concept of weak sensitivity in the mean for a subset $U\subseteq X^2$ and $(y,y')\in X^2$ we denote 
\[G_U(y,y'):=\{g\in G;\, (g.y,g.y')\in U\}\]
for the set of \emph{hitting times of $(y,y')$ to $U$} and consider a F\o lner sequence $\mathcal{F}=(F_n)_{n\in \mathbb{N}}$ in $G$. 

\begin{definition}\cite[Definition 5.4]{li2021mean}
    A pair $(x,x')\in X^2$ with $x\neq x'$ is called \emph{$\mathcal{F}$-weakly sensitive in the mean} 
    if for any open neighbourhood $U$ of $(x,x')$ there exists $c>0$ such that for all $\epsilon>0$ there are $n\in \mathbb{N}$ and $y,y'\in X$ with $d(y,y')<\epsilon$ 
    satisfying
    \begin{align*}
        \frac{\haar{G_U(y,y')\cap F_n}}{\haar{F_n}}>c. 
    \end{align*}
    We denote $\WSM^\mathcal{F}(X)$ for the set of all $\mathcal{F}$-weakly sensitive in the mean pairs. 
\end{definition}

We denote $\icerhull{Q}$ for the smallest closed invariant equivalence relation that contains a subset $Q\subseteq X^2$.
Building on the work of \cite{QiuZhao2020} it was shown in \cite[Corollary 5.6]{li2021mean} and \cite[Theorem 5.11]{li2021mean} that for the standard F\o lner sequence $\mathcal{F}$ we have that $(X,\mathbb{Z})$ is $\mathcal{F}$-mean equicontinuous if and only if $\WSM^\mathcal{F}(X)\subseteq \Delta_X$, and that $\RME^\mathcal{F}(X)=\icerhull{\WSM^\mathcal{F}(X)}$. 

It is natural to ask whether these results also hold in the context of actions of more general groups. 
We will see in Section \ref{sec:mainTheorem} that this is the case, whenever $G$ is Abelian. 
However, in Section \ref{sec:examples} below we will revisit an example from \cite{fuhrmann2025continuity} to show that beyond the context of actions of Abelian groups $\mathcal{F}$-weak sensitivity in the mean fails to describe $\RME^\mathcal{F}(X)$. We will present the following in Example~\ref{exa:lamplighterExample} below. 

\begin{example*}
    There exists a transitive action $(X,G)$ of a countable amenable group and a F{\o}lner sequence $\mathcal{F}$ in $G$ such that $(X,G)$ is $\mathcal{F}$-mean equicontinuous and admits a pair $(x,x')\in X^2$ with $x\neq x'$ that is $\mathcal{F}$-weakly sensitive in the mean. 
\end{example*}

In order to find an appropriate relation $Q$ that works in general let us reconsider the regionally proximal relation $\RP(X)$ for an action $(X,G)$. 
A pair $(x,x')\in X^2$ is called \emph{regionally proximal} if for any $\epsilon>0$ and any open neighbourhood $U$ of $(x,x')$ there exist $(y,y')\in U$ and $g\in G$ such that $d(g.y,g.y')<\epsilon$. 
To give an appropriate reformulation we call a sequence $(x_n,x_n')_{n\in \mathbb{N}}$ \emph{asymptotically diagonal}, whenever $d(x_n,x_n')\to 0$. 
A moment's thought reveals that a pair $(x,x')$ is regionally proximal if and only if for any open neighbourhood $U$ of $(x,x')$ there exists an asymptotically diagonal sequence $(x_n,x_n')_{n\in \mathbb{N}}$ such that 
\[\limsup_{n\to\infty}
    \haar{G_U(x_n,x_n')}>0.\]
Motivated by this observation for a (Borel-)measurable subset $G'\subseteq G$ and a F{\o}lner sequence $\mathcal{F}=(F_k)_{k\in \mathbb{N}}$ we consider the \emph{upper asymptotic density of $G'$ along $\mathcal{F}$} given by 
\[\upperAdens_\mathcal{F}(G'):=\limsup_{k\to \infty}\frac{\haar{G'\cap F_k}}{\haar{F_k}}\]
and define the following.  

\begin{definition}[$\mathcal{F}$-regionally mean sensitive pairs]
    Let $(X,G)$ be an action and $\mathcal{F}=(F_n)_{n\in \mathbb{N}}$ be a F{\o}lner sequence in $G$. 
    A pair $(x,x')\in X^2$ is called
    \emph{$\mathcal{F}$-regionally mean sensitive} if for any open neighbourhood $U$ of $(x,x')$ in $X^2$ 
    there exists an asymptotically diagonal sequence $(x_n,x_n')_{n\in \mathbb{N}}$ that satisfies
    \[
    \limsup_{n\to\infty}
    \upperAdens_\mathcal{F}(G_U(x_n,x_n'))
    >0.\]
    We denote $\RMS^\mathcal{F}(X)$ for the set of all $\mathcal{F}$-regionally mean sensitive pairs. 
\end{definition}

In Proposition \ref{pro:meanEquicontinuityViaRMS} we will see that an action $(X,G)$ is $\mathcal{F}$-mean equicontinuous if and only if $\RMS^\mathcal{F}(X)\subseteq \Delta_X$. 
Furthermore, we will see in Theorem \ref{the:maximalMeanEquicontinuousFactor} that the maximal $\mathcal{F}$-mean equicontinuous factor
can be described via the concept of $\mathcal{F}$-regional mean sensitivity, i.e.\ the following. 

\begin{theorem*}
\label{the:introRME_F}
    Let $(X,G)$ be an action of a locally compact and $\sigma$-compact amenable group on a compact metric space and $\mathcal{F}$ be a F{\o}lner sequence in $G$. We have $\RME^\mathcal{F}(X)=\icerhull{\RMS^\mathcal{F}(X)}$. 
\end{theorem*}

We define the \emph{(Weyl) mean pseudometric} 
$D:=\sup_\mathcal{F}D_\mathcal{F}$ by taking the supremum over all F{\o}lner sequences $\mathcal{F}$. 
An action $(X,G)$ is called \emph{mean equicontinuous} if $D\in C(X^2)$. 
Note that there exist actions that are $\mathcal{F}$-mean equicontinuous for some F{\o}lner sequence, but not mean equicontinuous \cite{fuhrmann2025continuity}. 
With a similar argument as above we observe that there exists a maximal mean equicontinuous factor $(X_{\text{me}},G)$ for any action $(X,G)$. 
Similarly, we denote $\pi_{\text{me}}\colon X\to X_{\text{me}}$ for the respective factor map and 
\[\RME(X):=\{(x,x')\in X^2;\, \pi_{\text{me}}(x)=\pi_{\text{me}}(x')\}\]
for the \emph{mean equicontinuous structure relation}.
Whenever $G$ is Abelian, or $(X,G)$ is minimal, then $(X_{\text{me}},G)$ and $(X_{\text{$\mathcal{F}$-me}},G)$ coincide for any F{\o}lner sequence $\mathcal{F}$ and our description so far is sufficient in order to study the maximal mean equicontinuous factor and $\RME(X)$.
However, for the general context we introduce the following relations. 
For this we denote
\[\upperBdens(G'):=\inf_{F}\sup_{g\in G}\frac{\haar{G'\cap Fg}}{\haar{F}}\]
for the \emph{upper Banach density} of a (Borel-)measurable subset $G'\subseteq G$, where the infimum is taken over all compact subsets $F\subseteq G$ with $\haar{F}>0$.  

\begin{definition}
    Let $(X,G)$ be an action.
    A pair $(x,x')\in X^2$ is called
    \begin{itemize}
        \item 
            \emph{regionally mean sensitive} if for any open neighbourhood $U$ of $(x,x')$ in $X^2$ 
            there exists an asymptotically diagonal sequence $(x_n,x_n')_{n\in \mathbb{N}}$ that satisfies
            \[\limsup_{n\to \infty}
            \upperBdens(G_U(x_n,x_n'))>0.\]
        \item 
            \emph{regionally joint mean sensitive} 
            if for any open neighbourhood $U$ of $(x,x')$ 
            there exists an asymptotically diagonal sequence $(x_n,x_n')_{n\in \mathbb{N}}$ and a F{\o}lner sequence $\mathcal{F}=(F_n)_{n\in \mathbb{N}}$ such that
            \[\limsup_{n\to \infty} \frac{\haar{G_U(x_n,x_n')\cap F_n}}{\haar{F_n}}>0.\]
    \end{itemize}
    We denote $\RMS(X)$ and $\SM(X)$ for the set of all regionally mean sensitive pairs and all regionally joint mean sensitive pairs, respectively. 
\end{definition}

We summarize our main results regarding these relations as follows. The details can be found in Theorem \ref{the:RMSInSMandFolnerRealisation} and Section \ref{sec:mainTheorem} below. 

\begin{theorem*}
\label{the:introRME_SM}
    Let $(X,G)$ be an action of a locally compact and $\sigma$-compact amenable group on a compact metric space. 
    \begin{enumerate}
    \item 
        For any F{\o}lner sequence $\mathcal{F}$ in $G$ we have $\RMS^\mathcal{F}(X)\subseteq \RMS(X)\subseteq \SM(X)$. 
    \item 
        The following statements are equivalent. 
        \begin{enumerate}
            \item $(X,G)$ is mean equicontinuous
            \item $\RMS(X)\subseteq \Delta_X$.
            \item $\SM(X)\subseteq \Delta_X$. 
        \end{enumerate}
    \item We have $\RME(X)=\icerhull{\RMS(X)}=\icerhull{\SM(X)}$.
    \end{enumerate}
\end{theorem*}

In the study of the $(\mathcal{F}-)$mean equicontinuity of an action $(X,G)$ an important role is played by the \emph{maximal support} (also \emph{measure center}) $\supp(X,G)$, which is defined as the union of all supports of invariant Borel probability measures \cite{li2014proximality, fuhrmann2025continuity}.
Note that $\supp(X,G)$ is a closed invariant subset of $X$ and that the amenability yields that this subset is non-empty. See Section \ref{sec:preliminaries} below for further details. 
An action is called \emph{fully supported} whenever $\supp(X,G)=X$. 

To name one application of this concept, it was shown in \cite{fuhrmann2022structure} that for fully supported actions the notion of $\mathcal{F}$-mean equicontinuity does not depend on the choice of the F{\o}lner sequence $\mathcal{F}$. 
In Section \ref{sec:reflexivity} we present that the maximal support is closely linked to the relations $\RMS^{(\mathcal{F})}$ and $\SM^{(\mathcal{F})}$. 
We prove that these relations are reflexive if and only if $X$ is fully supported. 
Furthermore, we show\footnote{
Note that a similar statement also holds for $\WSM^\mathcal{F}$ as presented in \cite[Proposition 5.14]{li2021mean} for the standard F\o lner sequence $\mathcal{F}$ in $\mathbb{Z}$.
} 
that these relations are contained in $\supp(X,G)^2$. 
We use this observation to conclude the following in Corollary \ref{cor:outsideSuppOneToOne}. 

\begin{corollary*}
    The factor map onto the maximal ($\mathcal{F}$-)mean equicontinuous factor is one-to-one on $X\setminus \supp(X,G)$.
\end{corollary*}

\subsection*{This article is organised as follows}
The next section collects the general background required throughout the article.
In Section~\ref{sec:interplayOfRelations}, we investigate the relations and
differences between $\RMS^{(\mathcal{F})}$, $\SM^{(\mathcal{F})}$, and $\WSM^{\mathcal{F}}$, and present some of their basic properties.
In Section~\ref{sec:characterisationViaMeasures}, we provide
characterisations of $\RMS^{(\mathcal{F})}$ and $\SM^{(\mathcal{F})}$ via invariant probability measures on $X^2$.
Building on this characterisation, we prove the results on the characterisation of the maximal ($\mathcal{F}-$)mean equicontinuous factor in Section~\ref{sec:mainTheorem}. 
Furthermore, in this section, we discuss the interplay of $\RMS$, $\SM$, the regional proximal relation $\RP$,  and the proximal relation $\P$. 
The details on the relation of the studied notions with the maximal support $\supp(X,G)$ are presented in Section~\ref{sec:reflexivity}. 
In the final section, we give full details on Example \ref{exa:lamplighterExample}. Furthermore, we provide examples that illustrate differences among the studied relations.

\subsection*{Convention}
Throughout this article, $G$ denotes a $\sigma$-compact locally compact amenable group, unless explicitly stated otherwise. 
We fix a (left) Haar measure on $G$, denoted by $\haar{\cdot}$.
Furthermore, in the following we equip each topological space with the respective Borel $\sigma$-algebra and speak of \emph{measurability} with respect to this $\sigma$-algebra. Furthermore, we simply speak of a \emph{probability measure}, whenever we mean a Borel probability measure.

We write $\RMS^{(\mathcal{F})}$ whenever we refer to $\RMS^{\mathcal{F}}$ and $\RMS$ at the same time. A similar abbreviation is applied for $\SM^{(\mathcal{F})}$.

\section{Preliminaries}
\label{sec:preliminaries}
For a set $X$ we denote $\Delta_X:=\{(x,x);\, x\in X\}$. 
We denote $C(X)$ for the Banach space of all continuous functions on a compact metric space $X$. 
We write $A\Delta B:=(A\setminus B) \cup (B\setminus A)$ for the symmetric difference of $A$ and $B$.

\subsection{Probability measures}
Let $X$ be a compact metric space. 
We denote by $\mathcal{M}(X)$ the set of all probability measures on $X$. 
Recall from the Riesz-Markov-Kakutani representation theorem that $\mathcal{M}(X)$ can be identified with a closed subset of the unit ball of the dual $C(X)^*$ of $C(X)$. 
This allows to equip $\mathcal{M}(X)$ with the weak*-topology.
It follows from the Banach-Alaoglu theorem that $\mathcal{M}(X)$ is a compact space. Furthermore, it is metrizable, which can be observed by considering the Wasserstein distance \cite[Chapter 7]{villani2003topics}.

For $\mu\in \mathcal{M}(X)$ we denote $\supp(\mu)$ for the \emph{support of $\mu$}, i.e.\ the set of all $x\in X$ such that any open neighbourhood $U$ of $x$ satisfies $\mu(U)>0$.
It follows from the invariance of $\mu$ that $\supp(\mu)$ is non-empty and invariant. 
Furthermore, it is a closed subset of $X$ as the complement of the union of all open sets $U\subseteq X$ with $\mu(U)=0$. 
For $x\in X$ we denote $\delta_x$ for the \emph{Dirac measure} at $x$, which is the unique probability measure with support $\{x\}$. 

Let $\pi\colon X\to Y$ be a continuous surjection between compact metric spaces. 
For $\mu\in \mathcal{M}(X)$ we denote $\pi_*\mu$ for the \emph{pushforward} of $\mu$ under $\pi$, which is defined by $\pi_*\mu(f):=\mu(f\circ \pi)$ for all $f\in C(Y)$. 
Note that $\pi_*$ establishes a continuous affine surjection $\pi_*\colon \mathcal{M}(X)\to \mathcal{M}(Y)$. 
Furthermore, for $\mu\in \mathcal{M}(X)$ we have 
$\supp(\pi_*\mu)=\pi(\supp(\mu))$ \cite[Theorem 2.4]{Okada1979}.

\subsection{F{\o}lner sequences and amenability}
Let $G$ be a $\sigma$-compact locally compact group. 
A sequence $(F_n)_{n\in \mathbb{N}}$ of compact subsets of $G$ with $\haar{F_n}>0$ is called \emph{(left) F{\o}lner} if for all non-empty compact subsets $K\subseteq G$ we have 
\[\lim_{n\to \infty} \haar{KF_n\Delta F_n}/\haar{F_n}\to 0.\]
$G$ is called \emph{amenable} if it allows for a F{\o}lner sequence.
A F{\o}lner sequence is called \emph{two-sided F{\o}lner} if for all non-empty compact subsets $K\subseteq G$ we additionally have 
$\lim_{n\to \infty} \haar{F_nK\Delta F_n}/\haar{F_n}\to 0$. 

\begin{remark}
    Whenever $G$ is countable, then a sequence of finite non-empty sets $(F_n)_{n\in \mathbb{N}}$ is F{\o}lner if and only if $\haar{gF_n\setminus F_n}/\haar{F_n}\to 0$ for all $g\in G$.
\end{remark}

\begin{remark}
    Any locally compact $\sigma$-compact Abelian group is amenable and allows for a two-sided F{\o}lner sequence. 
    However, not any locally compact $\sigma$-compact amenable group allows for a two-sided F\o lner sequence. 
    A locally compact $\sigma$-compact amenable (and separable) group $G$ allows for a two-sided F\o lner sequence if and only if it is unimodular.\footnote{Recall that a locally compact group is called \emph{unimodular} if $\haar{\cdot}$ is also right invariant. 
    Clearly, any locally compact Abelian group is unimodular.
    }
    This can be observed from combining the arguments from \cite[Appendix 3]{tempelman1992ergodic} with \cite[Proposition 2 (I.§1)]{ornstein1987entropy}.
\end{remark}

\begin{remark}
\label{rem:combiningFolnerFinite}
    Note that for a finite family of F{\o}lner sequences $(\mathcal{F}^{(i)})_{i=1}^n$ one can always find a F{\o}lner sequence $\mathcal{F}$ that has all $\mathcal{F}^{(i)}$ as a subsequence, for example by considering 
    $F_{kn+i}:=F_k^{(i)}$ for $k\in \mathbb{N}$ and $i\in \{1,\dots,n\}$. 
\end{remark}

For countable families the following holds. 

\begin{lemma}\cite[Proposition 2.1]{fuhrmann2025continuity}
\label{lem:combiningFolner}
For any countable family $(\mathcal{F}^{(i)})_{i\in \mathbb{N}}$ of F{\o}lner sequences there exists a F{\o}lner sequence $\mathcal{F}$ such that $\mathcal{F}$ and $\mathcal{F}^{(i)}$ have a common subsequence for all $i\in \mathbb{N}$. 
\end{lemma}

\subsection{Upper densities}
Consider a measurable subset $G'\subseteq G$. 
For a F{\o}lner sequence $\mathcal{F}=(F_n)$ we define the 
\emph{upper asymptotic density} as $\upperAdens_\mathcal{F}(G'):=\limsup_{n\to \infty}\haar{G'\cap F_n}/\haar{F_n}$. 
Furthermore, we define the \emph{upper Banach density} as 
$\upperBdens(G'):=\inf_{F}\sup_{g\in G}{\haar{G'\cap Fg}}/{\haar{F}}$, 
where the infimum is taken over all compact $F\subseteq G$ with $\haar{F}>0$. 
The following is well known and can be found in \cite[Lemma 2.8]{hauser2025mean}. 

\begin{lemma}
    \label{lem:upperBanachDensityCharacterisation}
    For any measurable $G'\subseteq G$ we have 
    \[
        \upperBdens(G')
        =\max_{\mathcal{F}}\upperAdens_{\mathcal{F}}(G'),
    \]
    where the maximum, taken over all  F{\o}lner sequences $\mathcal{F}$ in $G$, is attained. 
\end{lemma}

\subsection{Actions}
Let $G$ be a topological group and $X$ be a compact metric space. A \emph{(continuous) action} of $G$ on $X$ is a group homomorphism from $G$ into the group of homeomorphisms on $X$ such that the induced map $G\times X\to X, ~(g,x)\mapsto g.x$ is continuous.
Throughout the text we denote $(X,G)$ for an action and keep the group homomorphism implicit. 
We recommend \cite{auslander1988minimal} for a detailed exposition on actions. 

For $x\in X$, we denote $G.x:=\{g.x:g\in G\}$ for the \emph{orbit} of $x$. A point $x$ is called \emph{transitive} if $G.x$ is dense in $X$. An action $(X,G)$ is called \emph{transitive} if there exists a transitive point, and \emph{minimal} if each $x\in X$ is transitive. A subset $M\subseteq X$ is called \emph{invariant} if $M=g.M:=\{g.x;\, x\in M\}$ holds for all $g\in G$. 

A continuous surjection $\pi\colon X\to Y$ between $G$–actions is called a
\emph{factor map} if $\pi(g.x)=g.\pi(x)$ holds for all $(g,x)\in G\times X$. 
A factor map is called a \emph{conjugacy} if it is a homeomorphism and two actions are called \emph{conjugated} whenever there exists a conjugacy between them. 

On $X^2$ we consider the action of $G$ given by 
$g.(x,x'):=(g.x,g.x')$ for $g\in G$. 
A relation $Q\subseteq X^2$ is called \emph{invariant} if it is invariant w.r.t.\ this action. It is called \emph{closed} if it is a closed subset of $X^2$. 
Whenever $R$ is a closed invariant equivalence relation than $G$ induces an action on $X\big/G$ via $g.[x]_R:=[g.x]_R$.
Furthermore, note that the intersection of closed invariant equivalence relations is a closed invariant equivalence relation. For $Q\subseteq X^2$ we denote $\icerhull{Q}$ for the smallest closed invariant equivalence relation containing $Q$. 

For a factor map $\pi\colon X\to Y$ we denote $R(\pi):=\{(x,x');\, \pi(x)=\pi(x')\}$ for the respective \emph{fibre relation}. 
Note that $R(\pi)$ is a closed invariant equivalence relation and that $(Y,G)$ is conjugated to $(X\big/R(\pi),G)$.

\subsection{Cesàro averages}
Let $(X,G)$ be an action. 
For $g\in G$ and $f\in C(X)$ we denote 
$g^*f:=f\circ g$ for the \emph{pullback of $f$ under $g$}, where we identify $g$ with the induced homeomorphism. 
Furthermore, for $g\in G$ and $\mu\in \mathcal{M}_G(X)$ we denote $g_*\mu$ for the pushforward of $\mu$ under the homeomorphism induced by $g$. 
Note that $g^*f\in C(X)$ and $g_*\mu\in \mathcal{M}(X)$. 

Furthermore, for $F\subseteq G$ compact with $\haar{F}>0$ we denote 
\[F^*f:=\frac{1}{\haar{F}}\int_F (g^*f)\dd g
\hspace{1cm}\text{and}\hspace{1cm}
F_*\mu:=\frac{1}{\haar{F}}\int_F g_*\mu \dd g\]
for the respective \emph{Cesàro averages}. 
Note that 
$F^*f\in C(X)$
and that
$F_*\mu\in \mathcal{M}(X)$. 
Furthermore, we have $(F_*\mu)(f)=\mu(F^*f)$. 

\subsection{Invariant Borel probability measures and generic points}
Let $(X,G)$ be an action. 
We call $\mu\in \mathcal{M}(X)$ \emph{invariant} if $g_*\mu=\mu$ for all $g\in G$ and denote $\mathcal{M}_G(X)$ for the set of all invariant $\mu\in \mathcal{M}(X)$. 
A measure $\mu\in\mathcal{M}_G(X)$ is called \emph{ergodic} if for every invariant measurable set $M\subset X$ we have $\mu(M)\in\{0,1\}$. 
We denote by
$\mathcal{M}_G^e(X)$ the set of all ergodic $\mu\in \mathcal{M}_G(X)$.
If $\pi\colon X\to Y$ is a factor map, then we have $\pi_*\mu\in \mathcal{M}_G(Y)$ for all $\mu\in \mathcal{M}_G(X)$. 

For $x\in X$ and a F{\o}lner sequence $\mathcal{F}=(F_n)_{n\in \mathbb{N}}$ it follows from the compactness of $\mathcal{M}(X)$ that $((F_n)_*\delta_x)_{n\in \mathbb{N}}$ allows for cluster points. 
Furthermore, a standard Krylov-Bogolyubov argument yields that any such cluster point is invariant. 
For $x\in X$, a F{\o}lner sequence $\mathcal{F}=(F_n)_{n\in \mathbb{N}}$ and $\mu\in \mathcal{M}_G(X)$ we say that $x$ is \emph{$\mathcal{F}$-generic} for $\mu$ if $(F_n)_*\delta_x\to \mu$. 

\begin{lemma}{\cite{fuhrmann2022structure, fuhrmann2025continuity}}
\label{lem:genericPoints}
    Let $(X,G)$ be an action and $\mu\in \mathcal{M}_G(X)$ ergodic.
    \begin{enumerate}
        \item 
        Every left F{\o}lner sequence $\mathcal{F}$ allows for a subsequence $\mathcal{F}'$, such that 
    \linebreak
        $\mu$-almost every point $x$ is $\mathcal{F}'$-generic for $\mu$ \cite[Theorem 2.4]{fuhrmann2022structure}. 
        \item 
        Whenever $x\in X$ satisfies $\supp(\mu)\subseteq \overline{G.x}$ there exists a left F{\o}lner sequence $\mathcal{F}$ such that $x$ is $\mathcal{F}$-generic for $\mu$ \cite[Lemma 3.10]{fuhrmann2025continuity}. 
    \end{enumerate}
\end{lemma}

\subsection{The maximal support}
Let $(X,G)$ be an action. We denote 
\[\supp(X,G):=\bigcup_{\mu\in \mathcal{M}_G(X)}\supp(\mu)\]
for the \emph{maximal support} (also \emph{center/measure center}) of $(X,G)$.
An action $(X,G)$ is called \emph{fully supported} if $\supp(X,G)=X$. 

Recall that we assume $G$ to amenable. 
Thus, we observe $\supp(X,G)$ to be a non-empty invariant subset of $(X,G)$. 
Furthermore, it is well-known that $\supp(X,G)$ is closed.
This follows for example from the following lemma. 
We include a short proof for the reader’s convenience. 
See \cite{li2014proximality} and the references within for further reading on the maximal support. 

\begin{lemma}
\label{lem:characterisationMaximalSupport}
    Let $(X,G)$ be an action and $x\in X$. 
    We have $x\in \supp(X,G)$ if and only if for all open neighbourhoods $U$ of $x$ there exists $\mu\in \mathcal{M}_G(X)$ such that $\mu(U)>0$. 
\end{lemma}
\begin{proof}
    By the definition of the maximal support for $x\in \supp(X,G)$ there exists $\mu\in \mathcal{M}_G(X)$ such that $x\in \supp(\mu)$. 
    For any open neighbourhood $U$ of $x$ we observe that $\mu(U)>0$. 

    For the converse consider $x\in X$ such that for any open neighbourhood $U$ of $x$ there exists $\mu\in \mathcal{M}_G(X)$ such that $\mu(U)>0$. 
    Let $\mathcal{B}$ be a countable neighbourhood base for $x$ and enumerate $\mathcal{B}=\{B_n;\, n\in \mathbb{N}\}$. 
    For $n\in \mathbb{N}$ choose $\mu_n\in \mathcal{M}_G(X)$ with $\mu_n(B_n)>0$. 
    Denote $\mu:=\sum_{n=1}^\infty 2^{-n}\mu_n$ and note that $\mu\in \mathcal{M}_G(X)$. 
    To show that $x\in \supp(\mu)$ consider any open neighbourhood $U$ of $x$. 
    There exists $n\in \mathbb{N}$ such that $x\in B_n\subseteq U$ and we observe 
    $0<2^n\mu_n(B_n)\leq \mu(B_n)\leq \mu(U)$. We thus have that $x\in \supp(\mu)\subseteq \supp(X,G)$. 
\end{proof}

Considering product measures and pushforwards along the coordinate projections it is straightforward to obtain the following well-known statement.
See \cite[Proposition 3.12]{li2014proximality} for reference in the context of actions of $\mathbb{Z}$. 

\begin{lemma}
\label{lem:productSupport}
    For an action $(X,G)$ we have $\supp(X,G)^2=\supp(X^2,G)$. 
\end{lemma}

\subsection{Mean equicontinuity}
Let $(X,G)$ be an action. For a F{\o}lner sequence $\mathcal{F}=(F_n)_n$ we denote 
\[D_\mathcal{F}(x,x'):=\limsup_{n\to \infty} F_n^*d(x,x')\]
for the \emph{(Besicovitch) $\mathcal{F}$-mean pseudometric}. 
A point $x\in X$ is called \emph{$\mathcal{F}$-mean equicontinuous} if for any $\epsilon>0$ there exists a neighbourhood $U$ of $x$ such that 
$D_\mathcal{F}(x,x')<\epsilon$ for all $x'\in U$. 
An action is called \emph{$\mathcal{F}$-mean equicontinuous} if all $x\in X$ are $\mathcal{F}$-mean equicontinuous. 

\begin{remark}
    Exploring that $D_\mathcal{F}$ is a pseudometric it is straightforward to show that an action is $\mathcal{F}$-mean equicontinuous if and only if $D_\mathcal{F}\in C(X^2)$. 
\end{remark}

Taking the supremum over all F{\o}lner sequences $\mathcal{F}$ we denote $D:=\sup_\mathcal{F}D_\mathcal{F}$
for the \emph{(Weyl) mean pseudometric}. 
A point $x\in X$ is called \emph{mean equicontinuous} if for any $\epsilon>0$ there exists a neighbourhood $U$ of $x$ such that 
$D(x,x')<\epsilon$ for all $x'\in U$. 
An action is called \emph{mean equicontinuous} if all $x\in X$ are mean equicontinuous.
The main technique for the proof of the following lemma can be found in \cite[Proposition 3.3]{fuhrmann2025continuity} in the context of another pseudometric. We include the short proof for the convenience of the reader. 

\begin{lemma}
\label{lem:meanEquicontinuityCharacterisation}
    For an action $(X,G)$ the following statements are equivalent. 
    \begin{enumerate}
        \item \label{enu:MEC_meanEquicontinuousPoints}
        $(X,G)$ is mean equicontinuous. 
        \item \label{enu:MEC_Dcontinuous}
        $D\in C(X^2)$. 
        \item \label{enu:MEC_allFolner}
        $(X,G)$ is $\mathcal{F}$-mean equicontinuous for all F{\o}lner sequences $\mathcal{F}$ in $G$. 
    \end{enumerate}
\end{lemma}
\begin{proof}
    As above exploring that $D$ is a pseudometric yields the equivalence of 
    (\ref{enu:MEC_meanEquicontinuousPoints}) and (\ref{enu:MEC_Dcontinuous}). 
    Furthermore, from $D_\mathcal{F}\leq D$ it follows that (\ref{enu:MEC_Dcontinuous}) implies (\ref{enu:MEC_allFolner}). 
    To show that (\ref{enu:MEC_allFolner}) implies (\ref{enu:MEC_meanEquicontinuousPoints}) consider $x\in X$. 
    We need to show that $D(x,\cdot)\in C(X)$. 
    Since $X$ is metrizable it suffices to consider a sequence $(x_n)_{n\in \mathbb{N}}$ in $X$ that converges to $x$ and to show that $D(x_n,x)\to 0$. 
    For $n\in \mathbb{N}$ let $\mathcal{F}^{(n)}=(F_k^{(n)})_{k\in \mathbb{N}}$ be a F{\o}lner sequence, such that 
    $D(x_n,x)\leq \lim_{k\to \infty}(F_k^{(n)})^*d(x_n,x)+\tfrac{1}{n}$. 
    By Lemma \ref{lem:combiningFolner} there exists a F{\o}lner sequence $\mathcal{F}=(F_k)_{k\in \mathbb{N}}$ that shares a common subsequence with $\mathcal{F}^{(n)}$ for all $n\in \mathbb{N}$. 
    We thus observe that $D_\mathcal{F}(x_n,x)\leq D(x_n,x)\leq D_\mathcal{F}(x_n,x)+\tfrac{1}{n}$ holds for all $n\in \mathbb{N}$. 
    Since $(X,G)$ is $\mathcal{F}$-mean equicontinuous we thus observe that 
    $\lim_{n\to \infty}D(x_n,x)=\lim_{n\to \infty}D_\mathcal{F}(x_n,x)=0$. 
\end{proof}

\section{The interplay between the relations}
\label{sec:interplayOfRelations}
In this section we establish basic properties and study the relation between the concepts of regional mean sensitivity,
regional joint mean sensitivity, and 
weak sensitivity in the mean. 
Before starting the detailed discussion, we collect the relevant definitions for the convenience of the reader. 
For an elegant presentation of the subsequent proofs we first introduce the following. Recall that for $U\subseteq X^2$ and $(y,y')\in X^2$, we denote 
$G_U(y,y'):=\{g\in G;\, (g.y,g.y')\in U\}$ 
for the set of \emph{hitting times of $(y,y')$ to $U$}. 

\begin{definition}
\label{def:diagonallyLinkedAssociated}
    Let $(X,G)$ be an action. 
    We call a subset $U\subseteq X^2$ 
    \begin{itemize}
        \item 
            \emph{diagonally linked} 
            if there exists an asymptotically diagonal sequence 
            \linebreak 
            $(x_n,x_n')_{n\in \mathbb{N}}$ such that 
            \[\limsup_{n\to \infty} \upperBdens(G_U(x_n,x_n'))>0.\]
        \item 
            \emph{diagonally associated} 
            if there exists an asymptotically diagonal sequence $(x_n,x_n')_{n\in \mathbb{N}}$ and a F{\o}lner sequence $\mathcal{F}=(F_n)_{n\in \mathbb{N}}$ such that 
            \[\limsup_{n\to \infty}\frac{\haar{G_U(x_n,x_n')\cap F_n}}{\haar{F_n}}>0.\]
    \end{itemize}
    Whenever $\mathcal{F}=(F_n)_{n\in \mathbb{N}}$ is a F{\o}lner sequence in $G$ we call $U\subseteq X^2$
    \begin{itemize}
        \item 
            \emph{$\mathcal{F}$-diagonally linked} 
            if there exists an asymptotically diagonal sequence $(x_n,x_n')_{n\in \mathbb{N}}$ such that 
            \[\limsup_{n\to \infty}\upperAdens_\mathcal{F}(G_U(x_n,x_n'))>0.\]
        \item 
            \emph{$\mathcal{F}$-diagonally associated} 
            if there exists an asymptotically diagonal sequence $(x_n,x_n')_{n\in \mathbb{N}}$ such that 
            \[\limsup_{n\to \infty}\frac{\haar{G_U(x_n,x_n')\cap F_n}}{\haar{F_n}}>0.\]
    \end{itemize}
\end{definition}

The definitions given in the introduction can now be formulated as follows. 

\begin{definition}
\label{def:relations}
    Let $(X,G)$ be an action. 
    We call a pair $(x,x')\in X^2$
    \begin{itemize}    
    \item 
        \emph{regionally mean sensitive} 
        [$(x,x')\in \RMS(X)$]
        if any open neighbourhood $U$ of $(x,x')$ is diagonally linked. 
    
    \item 
        \emph{regionally joint mean sensitive}
        [$(x,x')\in \SM(X)$]
        if any open neighbourhood $U$ of $(x,x')$ 
        is diagonally associated. 
        
    \end{itemize}
    Furthermore, for a F{\o}lner sequence $\mathcal{F}=(F_n)_{n\in \mathbb{N}}$ we call a pair $(x,x')\in X^2$
    \begin{itemize}
    \item
        \emph{$\mathcal{F}$-regionally mean sensitive} 
        [$(x,x')\in \RMS^\mathcal{F}(X)$]
        if any open neighbourhood $U$ of $(x,x')$
        is $\mathcal{F}$-diagonally linked. 
    
    \item
        \emph{$\mathcal{F}$-regionally joint mean sensitive} 
        [$(x,x')\in \SM^\mathcal{F}(X)$]
        if any open neighbourhood $U$ of $(x,x')$ 
        is $\mathcal{F}$-diagonally associated. 
    
    \item 
        \emph{$\mathcal{F}$-weakly sensitive in the mean} 
        [$(x,x')\in \WSM^\mathcal{F}(X)$]
        if $x\neq x'$ and for any open neighbourhood $U$ of $(x,x')$ there exists $c>0$ such that for all $\epsilon>0$ there are $n\in \mathbb{N}$ and $y,y'\in X$ with $d(y,y')<\epsilon$ such that 
        ${\haar{G_U(y,y')\cap F_n}}/{\haar{F_n}}>c. $
    \end{itemize}
\end{definition}

\begin{remark}
    The reader might wonder why we are not excluding diagonal elements in the definitions of $\RMS(X)$, $\SM(X)$, $\RMS^\mathcal{F}(X)$ and $\SM^\mathcal{F}(X)$. 
    Our choice allows for the study of the maximal support, which we will present in Section \ref{sec:reflexivity} below.
\end{remark}

\begin{remark}
\label{rem:diagonallyLinkedAssociatedMonotonicity}
    Note that whenever $U$ is diagonally linked and contained in some $V$, then also $V$ is diagonally linked. 
    In particular, we observe that a pair is regionally mean sensitive if and only if it allows for a neighbourhood base of diagonally linked sets. 
    Clearly, similar statements also hold for the other notions introduced in the Definitions \ref{def:diagonallyLinkedAssociated} and \ref{def:relations}. 
\end{remark}

\subsection{The interplay of $\SM^{\mathcal{F}}$ and $\WSM^{\mathcal{F}}$}
In this subsection we show that for non-diagonal pairs the concepts of $\mathcal{F}$-regionally joint mean sensitivity and $\mathcal{F}$-weak sensitivity in the mean coincide.
For this we will need the following. 

\begin{lemma}
\label{lem:diagonallyLinkedAssociatedSubsequenceStability}
    Let $\mathcal{F}$ be a F{\o}lner sequence and $\mathcal{F}'$ be a subsequence of $\mathcal{F}$. 
    \begin{enumerate}
        \item \label{enu:DLASS_associated}
            If $U\subseteq X^2$ is $\mathcal{F}'$-diagonally associated, then it is $\mathcal{F}$-diagonally associated. 
        \item \label{enu:DLASS_linked}
            If $U\subseteq X^2$ is $\mathcal{F}'$-diagonally linked, then it is $\mathcal{F}$-diagonally linked.
    \end{enumerate}
\end{lemma}
\begin{proof}
(\ref{enu:DLASS_associated}):
    Denote $\mathcal{F}=(F_n)_{n\in \mathbb{N}}$ and $\mathcal{F}'=(F_k')_{k\in \mathbb{N}}$. 
    Let $(y_k,y_k')_{k\in \mathbb{N}}$ be an asymptotically diagonal sequence that satisfies 
    \[\limsup_{k\to \infty}\frac{\haar{G_U(y_k,y_k')\cap F_k'}}{\haar{F_k'}}>0.\]
    Furthermore, let $(n_k)_{k\in \mathbb{N}}$ be a strictly increasing sequence in $\mathbb{N}$ such that $F_{n_k}=F_k'$ for all $k\in \mathbb{N}$. 
    Denote $n_0:=0$. 
    For $n\in \mathbb{N}$ and $k\in \mathbb{N}$ with $n_{k-1}<n\leq n_k$ we denote
    $x_n:=y_{n_k}$ and 
    $x_n':=y_{n_k}'$. 
    Clearly, $(x_n,x_n')_{n\in \mathbb{N}}$ is an asymptotically diagonal sequence that satisfies $(x_{n_k},x_{n_k}')=(y_k,y_k')$ for all $k\in \mathbb{N}$. 
    Thus, we have
    \begin{align*}
        \limsup_{n\to \infty} \frac{\haar{G_U(x_n,x_n')\cap F_{n}}}{\haar{F_{n}}}
        &\geq \limsup_{k\to \infty} \frac{\haar{G_U(x_{n_k},x_{n_k}')\cap F_{n_k}}}{\haar{F_{n_k}}}\\
        &=\limsup_{k\to \infty} \frac{\haar{G_U(y_{k},y_{k}')\cap F_{k}'}}{\haar{F_{k}'}}
        >0.
    \end{align*}
    
(\ref{enu:DLASS_linked}): 
    This is immediate from the definition, since for any measurable $G'\subseteq G$ we have 
    $\upperAdens_{\mathcal{F}'}(G')\leq \upperAdens_{\mathcal{F}}(G').$
\end{proof}

Lemma \ref{lem:diagonallyLinkedAssociatedSubsequenceStability} allows the following insight. 

\begin{lemma}
\label{lem:literatureDock}
    Let $(X,G)$ be an action and $\mathcal{F}$ a F{\o}lner sequence in $G$. For any $U\subseteq X^2$ with $\Delta_X\cap \overline{U}=\emptyset$ the following statements are equivalent. 
    \begin{enumerate}
        \item \label{enu:LD_DiagonallyAssociated}
        $U$ is $\mathcal{F}$-diagonally associated. 
        \item \label{enu:LD_LiYu}
        There exists $c>0$ such that 
        for any $\epsilon>0$
        there exist $y,y'\in X$ and $n\in \mathbb{N}$ such that $d(y,y')<\epsilon$ and 
        ${\haar{G_U(y,y')\cap F_n}}/{\haar{F_n}}\geq c.$
        \item \label{enu:LD_LiYuTwo}
        There exists $c>0$ such that 
        for any $\epsilon>0$ and any $N\in \mathbb{N}$
        there exist $y,y'\in X$ and $n\geq N$ such that $d(y,y')<\epsilon$ and 
        ${\haar{G_U(y,y')\cap F_n}}/{\haar{F_n}}\geq c.$
    \end{enumerate}
\end{lemma}
\begin{remark}
    The argument for the equivalence of (\ref{enu:LD_LiYu}) and (\ref{enu:LD_LiYuTwo}) is essentially contained in \cite[Lemma 5.2]{li2021mean}, where it is presented for the standard F{\o}lner sequence in $\mathbb{Z}$. 
    We include the short proof for the convenience of the reader. 
\end{remark}
\begin{proof}
(\ref{enu:LD_DiagonallyAssociated})$\Rightarrow$(\ref{enu:LD_LiYu}):
    Trivial. 

(\ref{enu:LD_LiYu})$\Rightarrow$(\ref{enu:LD_LiYuTwo}):
    Let $c>0$. 
    If (\ref{enu:LD_LiYuTwo}) does not hold then there exist $N_c\in \mathbb{N}$ and $\epsilon_c>0$ such that 
    for all $(y,y')\in X^2$ with $d(y,y')<\epsilon_c$ and all $n\geq N_c$ we have 
    \[\frac{\haar{G_U(y,y')\cap F_n}}{\haar{F_n}}<c.\]
    Note that $F:=\bigcup_{n=1}^{N_c} F_n$ is compact and hence that 
    $F\times X\to X, (g,x)\mapsto g.x$
    is uniformly continuous. 
    Since $\overline{U}\cap \Delta_X=\emptyset$ we find $\epsilon>0$ such that for all $(y,y')\in X^2$ with $d(y,y')<\epsilon$ we have $(g.y,g.y')\notin U$ for all $g\in F$. 
    Considering $\epsilon_c':=\min\{\epsilon_c,\epsilon\}$ we observe that for $(y,y')\in X^2$ with $d(y,y')<\epsilon_c'$ and all $n\in \mathbb{N}$ we have 
    \[\frac{\haar{G_U(y,y')\cap F_n}}{\haar{F_n}}< c.\]
    Thus (\ref{enu:LD_LiYu}) is not satisfied. 

(\ref{enu:LD_LiYuTwo})$\Rightarrow$(\ref{enu:LD_DiagonallyAssociated}):
    By (\ref{enu:LD_LiYuTwo}) there exists an asymptotically diagonal sequence $(x_k,x_k')_{k\in \mathbb{N}}$ and a strictly increasing sequence $(n_k)_{k\in \mathbb{N}}$ in $\mathbb{N}$ such that 
    \[\inf_{k\in \mathbb{N}}\frac{\haar{G_U(x_k,x_k')\cap F_{n_{k}}}}{\haar{F_{n_{k}}}}\geq c.\]
    Denoting $\mathcal{F}':=(F_{n_k})_{k\in \mathbb{N}}$ we have found a subsequence of $\mathcal{F}$ such that $U$ is $\mathcal{F}'$-diagonally associated. 
    It follows from Lemma \ref{lem:diagonallyLinkedAssociatedSubsequenceStability} that $U$ is $\mathcal{F}$-diagonally associated. 
\end{proof}

We are now prepared to show the claimed statement about regional joint mean sensitivity and weak sensitivity in the mean. 

\begin{proposition}
\label{pro:interplayweaksitmAndregionallysitm}
    Let $(X,G)$ be an action and $\mathcal{F}$ be a F{\o}lner sequence in $G$. 
    We have $\SM^\mathcal{F}(X)\setminus \Delta_X=\WSM^\mathcal{F}(X)$. 
    In particular, we have $\icerhull{\SM^\mathcal{F}(X)}=\icerhull{\WSM^\mathcal{F}(X)}$. 
\end{proposition}
\begin{proof}
    Let $(x,x')\in X^2$. 
    Recall that $x\neq x'$ is part of the definition of $\WSM^\mathcal{F}$. We thus assume, w.l.o.g.\ that $x\neq x'$. 
    Combining the observation from Remark \ref{rem:diagonallyLinkedAssociatedMonotonicity} with Lemma \ref{lem:literatureDock} it is then straightforward to observe that $(x,x')\in \SM^\mathcal{F}(X)$ if and only if $(x,x')\in \WSM^\mathcal{F}(X)$. 
\end{proof}

At this point it is natural to ask whether removing the assumption $x\neq x'$ from the definition of $\mathcal{F}$-weak sensitivity in the mean yields that $\WSM^\mathcal{F}(X)$ and $\SM^\mathcal{F}(X)$ coincide also on the diagonal. The following example illustrates that this is not the case. 

\begin{example}
\label{exa:literatureDock}
    Let $X=\mathbb{Z}\cup \{\infty\}$ be the one point compactification of $\mathbb{Z}$. We act with $G=\mathbb{Z}$ on $X$ by fixing $\infty$ and mapping $g.x:=g+x$ for $g,x\in \mathbb{Z}$. 
    Furthermore, we consider the standard F{\o}lner sequence $\mathcal{F}=(F_n)_{n\in \mathbb{N}}$ given by 
    $F_n:=\{0,\dots,n-1\}$. 

    To show that $\SM^\mathcal{F}(X)\cap \Delta_X=\{(\infty,\infty)\}$ first consider $x\in X\setminus \{\infty\}$ and 
    note that $x$ is an isolated point of $X$. 
    Thus $U:=\{(x,x)\}$ is an open neighbourhood of $(x,x)$ in $X^2$. 
    Let $(x_n,x_n')_{n\in \mathbb{N}}$ be an asymptotically diagonal sequence in $X^2$. 
    For $n\in \mathbb{N}$ we observe that 
    $\haar{G_U(x_n,x_n')}\leq 1$ and hence that 
    $\limsup_{n\to \infty}{\haar{G_U(x_n,x_n')\cap F_n}}/{\haar{F_n}}=0$. 
    Thus $U$ is not $\mathcal{F}$-diagonally associated and hence $(x,x)\notin \SM^\mathcal{F}(X)$. 
    Furthermore, for $x=\infty$ we know that $x$ is fixed and hence that $G_U(x,x)=G$ for all open neighbourhoods $U$ of $(x,x)$. 
    Thus any open neighbourhood of $(x,x)$ is diagonally associated. 
    This shows that 
    \[\SM^\mathcal{F}(X)\cap \Delta_X=\{(\infty,\infty)\}.\] 

    However, for any $x\in X$ and an open neighbourhood $U$ of $(x,x)$ we observe that for any $\epsilon>0$ we have 
    \[\frac{\haar{G_U(x,x)\cap F_1}}{\haar{F_1}}=\frac{\haar{\{0\}}}{\haar{\{0\}}}=1>0.\]
    Thus, changing the definition of $\mathcal{F}$-weak sensitivity in the mean by removing the condition `$x\neq x'$' would yield that all diagonal pairs are $\mathcal{F}$-weakly sensitive in the mean for this example. 
\end{example}

\subsection{Basic properties of $\RMS^{(\mathcal{F})}$ and $\SM^{(\mathcal{F})}$}

We next collect some basic properties of these relations, which will become important in the subsequent discussion.  
Recall that we denote $\RP(X)$ for the regional proximal relation. 

\begin{proposition}
\label{pro:basicPropertiesRMSandSM}
    Let $(X,G)$ be an action and $\mathcal{F}$ be a F{\o}lner sequence.  
    \begin{enumerate}
        \item \label{enu:BPRMSSM_inclusion}
            $\RMS^\mathcal{F}(X)\subseteq \RMS(X)\subseteq \RP(X)$
            and 
            $\SM^\mathcal{F}(X)\subseteq \SM(X)\subseteq  \RP(X)$. 
        \item \label{enu:BPRMSSM_monotonicity}
            Whenever $\mathcal{F}'$ is a subsequence of $\mathcal{F}$, then 
            \[\RMS^{\mathcal{F}'}(X)\subseteq\RMS^\mathcal{F}(X)
            \hspace{1cm}\text{and}\hspace{1cm}
            \SM^{\mathcal{F}'}(X) \subseteq \SM^\mathcal{F}(X).\]
        \item 
            $\RMS^\mathcal{F}(X)$ and $\SM^\mathcal{F}(X)$ are closed and symmetric  relations.
        \item 
            $\RMS(X)$ and $\SM(X)$ are closed, invariant and symmetric relations.
    \end{enumerate}
\end{proposition}
\begin{proof}
    Note that (\ref{enu:BPRMSSM_monotonicity}) follows from Lemma \ref{lem:diagonallyLinkedAssociatedSubsequenceStability}. 
    The remaining statements can easily be deduced from the definitions. 
\end{proof}

\begin{remark}
    We will see in Section~\ref{sec:reflexivity} that the relations discussed in Proposition \ref{pro:basicPropertiesRMSandSM} (except $\RP(X)$) are reflexive if and only if $X$ is fully supported. Furthermore, as presented in Example~\ref{exa:twoPointCompactificationThreeCopies} below, none of these relations needs to be transitive. 
\end{remark}

\begin{remark}
    The inclusions in (\ref{enu:BPRMSSM_inclusion}) and (\ref{enu:BPRMSSM_monotonicity}) can be strict as illustrated by Example~\ref{exa:twoPointCompactificationThreeCopies} below.
\end{remark}

With a similar argument as in the proof of (\ref{enu:LD_LiYuTwo})$\Rightarrow$(\ref{enu:LD_DiagonallyAssociated}) in Lemma \ref{lem:literatureDock} it follows that the limit superior in the definition of regional (joint) mean sensitivity can be replaced by an infimum.
We present the details of this observation in the following proposition. 

\begin{proposition}
\label{pro:infimumReformulation}
    Let $(X,G)$ be an action and $U\subseteq X^2$. 
    \begin{enumerate}
    \item \label{enu_IR_diagonallyLinked}
        $U$ is diagonally linked
            if and only if 
        there exists an asymptotically diagonal sequence $(x_n,x_n')_{n\in \mathbb{N}}$ that satisfies
        $\inf_{n\in \mathbb{N}} \upperBdens(G_U(x_n,x_n')) > 0.$
    \item \label{enu_IR_diagonallyAssociated}
        $U$ is diagonally associated 
        if and only if 
        there exists an asymptotically diagonal sequence $(x_n,x_n')_{n\in \mathbb{N}}$ and a F{\o}lner sequence $(F_n)_{n\in \mathbb{N}}$ such that 
        $\inf_{n\in \mathbb{N}} {\haar{G_U(x_n,x_n')\cap F_n}}/{\haar{F_n}}>0.$
    \end{enumerate}
Let $\mathcal{F}$ be a F{\o}lner sequence in $G$. 
    \begin{enumerate}
    \setcounter{enumi}{2}
    \item \label{enu_IR_diagonallyLinked_F}
        $U$ is $\mathcal{F}$-diagonally linked
            if and only if 
        there exists an asymptotically diagonal sequence $(x_n,x_n')_{n\in \mathbb{N}}$ that satisfies
        $\inf_{n\in \mathbb{N}}
        \upperAdens_\mathcal{F}(G_U(x_n,x_n')) > 0.$
    \item \label{enu_IR_diagonallyAssociated_F}
        $U$ is $\mathcal{F}$-diagonally associated
            if and only if 
        there exists an asymptotically diagonal sequence $(x_n,x_n')_{n\in \mathbb{N}}$ and a subsequence $(F_n')_{n\in \mathbb{N}}$ of $\mathcal{F}$ such that 
        $\inf_{n\in \mathbb{N}} {\haar{G_U(x_n,x_n')\cap F_n'}}/{\haar{F_n'}}>0.$
    \end{enumerate}
\end{proposition}

\subsection{The interplay of $\RMS^{(\mathcal{F})}$ and $\SM^{(\mathcal{F})}$}
In this subsection we provide details on the relationship between $(\mathcal{F})$-regional mean sensitivity and $(\mathcal{F}-)$regional joint mean sensitivity. 
We summarize our results in the following theorem, and provide the proof after a short discussion. 

\begin{theorem}
\label{the:RMSInSMandFolnerRealisation}
    Let $(X,G)$ be an action. 
    \begin{enumerate}
        \item \label{enu:RMSISMAFR_inclusion_F}
            $\RMS^\mathcal{F}(X)\subseteq \SM^\mathcal{F}(X)$ holds for any F{\o}lner sequence $\mathcal{F}$. 
       
        \item \label{enu:RMSISMAFR_realization}
            There exists a F{\o}lner sequence $\mathcal{F}$ such that 
            \[
            \RMS(X)= \RMS^\mathcal{F}(X)
            \hspace{1cm}\text{and}\hspace{1cm}
            \SM(X)=\SM^\mathcal{F}(X) 
            .\]
            
        \item \label{enu:RMSISMAFR_inclusion}
            $\RMS(X)\subseteq \SM(X)$. 
    \end{enumerate}
\end{theorem}
\begin{remark}
    Note that the F{\o}lner sequence in (\ref{enu:RMSISMAFR_realization}) depends on the action $(X,G)$. 
\end{remark}

\begin{remark}
    The inclusion in (\ref{enu:RMSISMAFR_inclusion_F}) 
    can be strict as we will see in Example \ref{exa:lamplighterExample} below. 
    However, it remains open whether the inclusion in (\ref{enu:RMSISMAFR_inclusion}) can be strict. 
\end{remark}

Before presenting the proof of Theorem \ref{the:RMSInSMandFolnerRealisation}, we highlight the following consequence of Theorem \ref{the:RMSInSMandFolnerRealisation}(\ref{enu:RMSISMAFR_realization}).

\begin{corollary}
\label{cor:RMSvsRMSFforAllF}
    Let $(X,G)$ be an action and $(x,x')\in X^2$.  
    \begin{enumerate}
        \item 
            $(x,x')$ is regionally mean sensitive 
            if and only if 
            it is $\mathcal{F}$-regionally mean sensitive w.r.t.\ some F{\o}lner sequence $\mathcal{F}$ in $G$. 
        \item 
            $(x,x')$ is regionally joint mean sensitive
            if and only if 
            it is $\mathcal{F}$-regionally joint mean sensitive w.r.t.\ some F{\o}lner sequence $\mathcal{F}$ in $G$. 
    \end{enumerate}
\end{corollary}

For the proof of Theorem \ref{the:RMSInSMandFolnerRealisation} we first show the following. 

\begin{lemma}
\label{lem:RMS_SMbridge}
    Let $(X,G)$ be an action and $\mathcal{F}$ be a F{\o}lner sequence. 
    Any open $U\subseteq X^2$ that is $\mathcal{F}$-diagonally linked is $\mathcal{F}$-diagonally associated. 
\end{lemma}
\begin{proof}
    Denote $\mathcal{F}=(F_k)_{k\in \mathbb{N}}$. 
    By Proposition \ref{pro:infimumReformulation} there exists an asymptotically diagonal sequence $(x_n,x_n')_{n\in \mathbb{N}}$ that satisfies 
    \[c:=\tfrac{1}{2}\inf_{n\in \mathbb{N}} \limsup_{k\to \infty}\frac{\haar{G_U(x_n,x_n')\cap F_{k}}}{\haar{F_k}}>0.\]
    For $n=1$, we choose $k_1\in\mathbb{N}$ such that 
    \[\frac{\haar{G_U(x_1,x_1')\cap F_{k_1}}}{\haar{F_{k_1}}}\geq c.\]
    Proceeding inductively, we choose a strictly increasing sequence $(k_n)_{n\in \mathbb{N}}$ such that for each $n\in\mathbb{N}$ we have
    \[\frac{\haar{G_U(x_n,x_n')\cap F_{k_n}}}{\haar{F_{k_n}}}\geq c.\]
    Considering the subsequence $\mathcal{F}':=(F_{k_n})_{n\in \mathbb{N}}$
    we observe
    \[
    \limsup_{n\to \infty}\frac{\haar{G_U(x_n,x_n')\cap F_{k_n}}}{\haar{F_{k_n}}}
    \geq c>0.
    \]
    This shows that $U$ is $\mathcal{F}'$-diagonally associated and from Lemma \ref{lem:diagonallyLinkedAssociatedSubsequenceStability} we observe that $U$ is $\mathcal{F}$-diagonally associated. 
\end{proof}

\begin{lemma}
\label{lem:diagonallyLinkedAndAssociatedFindFolner}
    Let $(X,G)$ be an action and $\mathcal{B}$ a countable family of subsets of $X^2$. 
    \begin{enumerate}
        \item \label{enu:DLAFF_linked}
            If all $U\in\mathcal{B}$ are diagonally linked, then there exists a F{\o}lner sequence $\mathcal{F}$ such that all $U\in \mathcal{B}$ are $\mathcal{F}$-diagonally linked. 

        \item \label{enu:DLAFF_associated}
            If all $U\in\mathcal{B}$ are diagonally associated, then there exists a F{\o}lner sequence $\mathcal{F}$ such that all $U\in \mathcal{B}$ are $\mathcal{F}$-diagonally associated. 
    \end{enumerate}
\end{lemma}
\begin{proof}
(\ref{enu:DLAFF_linked}):
    For each $U\in \mathcal{B}$, Proposition \ref{pro:infimumReformulation} allows one to choose an asymptotically diagonal sequence $(x_n^U,y_n^U)$ with 
    \[c:=\inf_{n\in \mathbb{N}} \upperBdens(G_U(x_n^U,y_n^U))>0.\]
    Furthermore, for $n\in \mathbb{N}$, Lemma \ref{lem:upperBanachDensityCharacterisation} allows one to choose a F{\o}lner sequence $\mathcal{F}^{(U,n)}=(F_k^{(U,n)})_{k\in \mathbb{N}}$ such that 
    \[\upperBdens(G_U(x_n^U,y_n^U))=\lim_{k\to \infty} \frac{\haar{G_U(x_n^U,y_n^U)\cap F_k^{(U,n)}}}{\haar{F_k^{(U,n)}}}.\]
    
    Consider the countable family of F{\o}lner sequences $\{\mathcal{F}^{(U,n)};\, (U,n)\in \mathcal{B}\times \mathbb{N}\}$. 
    By Lemma \ref{lem:combiningFolner} there exists a F{\o}lner sequence $\mathcal{F}=(F_k)_{k\in \mathbb{N}}$ that shares a common subsequence with $\mathcal{F}^{(U,n)}$ for all $(U,n)\in \mathcal{B}\times \mathbb{N}$.

    To show that each element of $\mathcal{B}$ is $\mathcal{F}$-diagonally linked consider $U\in \mathcal{B}$. 
    Consider a strictly increasing sequence $(k_i)_{i\in \mathbb{N}}$ in $\mathbb{N}$ such that $\mathcal{F}':=(F_{k_i}^{(U,n)})_{i\in \mathbb{N}}$ is a subsequence of $\mathcal{F}$. 
    For all $n\in \mathbb{N}$ we observe that 
    \begin{align*}
        \upperAdens_{\mathcal{F}'}(G_U(x_n^U,y_n^U))
        &=\limsup_{i\to \infty} \frac{\haar{G_U(x_n^U,y_n^U)\cap F_{k_i}^{(U,n)}}}{\haar{F_{k_i}^{(U,n)}}}\\
        &=\lim_{k\to \infty} \frac{\haar{G_U(x_n^U,y_n^U)\cap F_k^{(U,n)}}}{\haar{F_k^{(U,n)}}}\\
        &=\upperBdens(G_U(x_n^U,y_n^U))\geq c.
    \end{align*}
    Thus, the asymptotically diagonal sequence $(x_n^U,y_n^U)_{n\in \mathbb{N}}$ satisfies 
    \[\limsup_{n\to \infty}\upperAdens_{\mathcal{F}'}(G_U(x_n^U,y_n^U))
    \geq c>0.\]
    This shows that $U$ is $\mathcal{F}'$-diagonally linked. 
    Since $\mathcal{F}'$ is a subsequence of $\mathcal{F}$ we observe $U$ to be $\mathcal{F}$-diagonally linked from Lemma \ref{lem:diagonallyLinkedAssociatedSubsequenceStability}. 

(\ref{enu:DLAFF_associated}):
    For each $U\in \mathcal{B}$ choose an asymptotically diagonal sequence $(x_n^U,y_n^U)$ and a F{\o}lner sequence $\mathcal{F}^U=(F_n^U)_{n\in\mathbb{N}}$ such that
    \[\lim_{n\to \infty} \frac{\haar{G_{U}(x_n^U,y_n^U)\cap F_n^U}}{\haar{F_n^U}}>0.\]
    From Lemma \ref{lem:combiningFolner} we know that there exists a F{\o}lner sequence $\mathcal{F}=(F_n)_n$ that shares a common subsequence with $\mathcal{F}^{U}$ for all $U\in \mathcal{B}$.

    To show that each element of $\mathcal{B}$ is $\mathcal{F}$-diagonally associated consider $U\in \mathcal{B}$. 
    Since $\mathcal{F}$ and $\mathcal{F}^U$ share a common subsequence there exists a strictly increasing sequence $(n_k)_{k\in \mathbb{N}}$ such that $\mathcal{F}':=(F_{n_k}^U)_{k\in \mathbb{N}}$ is a subsequence of $\mathcal{F}$. 
    Clearly, we have 
    \begin{align*}
        \lim_{k\to \infty} \frac{\haar{G_{U}(x_{n_k}^U,y_{n_k}^U)\cap F_{n_k}^U}}{\haar{F_{n_k}^U}}
        =
        \lim_{n\to \infty} \frac{\haar{G_{U}(x_n^U,y_n^U)\cap F_n^U}}{\haar{F_n^U}}>0.
    \end{align*}
    Since $(x_{n_k}^U,y_{n_k}^U)_{k\in \mathbb{N}}$ is asymptotically diagonal we observe that $U$ is $\mathcal{F}'$-diagonally associated. 
    Since $\mathcal{F}'$ is a subsequence of $\mathcal{F}$ we observe from Lemma \ref{lem:diagonallyLinkedAssociatedSubsequenceStability} that $U$ is $\mathcal{F}$-diagonally associated. 
\end{proof}

\begin{proof}[Proof of Theorem \ref{the:RMSInSMandFolnerRealisation}:]
(\ref{enu:RMSISMAFR_inclusion_F}): 
    Consider a F{\o}lner sequence $\mathcal{F}$ and $(x,x')\in \RMS^\mathcal{F}(X)$. 
    Any open neighbourhood $U$ of $(x,x')$ is $\mathcal{F}$-diagonally linked, and hence 
\linebreak    
    $\mathcal{F}$-diagonally associated by Lemma \ref{lem:RMS_SMbridge}. 
    We thus observe $(x,x')\in \SM^\mathcal{F}(X)$. 

(\ref{enu:RMSISMAFR_realization}):
    Let $\mathcal{B}$ be a countable base for the topology of $X$ and denote $\mathcal{B}^l$ and $\mathcal{B}^a$ for the sets of all diagonally linked $U\in \mathcal{B}$ and diagonally associated $U\in \mathcal{B}$, respectively. 
    It follows from Lemma \ref{lem:diagonallyLinkedAndAssociatedFindFolner} that there exist F{\o}lner sequences $\mathcal{F}^l$ and $\mathcal{F}^a$ such that all $U\in \mathcal{B}^l$ are $\mathcal{F}^l$-diagonally linked and all $U\in \mathcal{B}^a$ are $\mathcal{F}^a$-diagonally associated. 
    In consideration of Remark \ref{rem:combiningFolnerFinite} we choose a F{\o}lner sequence $\mathcal{F}$ that has $\mathcal{F}^l$ and $\mathcal{F}^a$ as subsequences. 
    It follows from Lemma \ref{lem:diagonallyLinkedAssociatedSubsequenceStability} that all $U\in \mathcal{B}^l$ are $\mathcal{F}$-diagonally linked and that all $U\in \mathcal{B}^a$ are $\mathcal{F}$-diagonally associated. 

    To show that $\RMS(X)=\RMS^\mathcal{F}(X)$ consider $(x,x')\in \RMS(X)$. 
    Any neighbourhood of $(x,x')$ is diagonally linked and hence there exists a neighbourhood base $\mathcal{B}'$ for $(x,x')$ with $\mathcal{B}'\subseteq \mathcal{B}^l$.
    Since any $U\in \mathcal{B}^l$ is $\mathcal{F}$-diagonally linked we observe $(x,x')\in \RMS^\mathcal{F}(X)$ from Remark \ref{rem:diagonallyLinkedAssociatedMonotonicity}.
    A similar argument shows that $\SM(X)=\SM^\mathcal{F}(X)$. 

(\ref{enu:RMSISMAFR_inclusion}):
    It follows from (\ref{enu:RMSISMAFR_realization}) that there exists a F{\o}lner sequence $\mathcal{F}$ such that $\RMS^\mathcal{F}(X)=\RMS(X)$ and $\SM^\mathcal{F}(X)=\SM(X)$.
    We thus observe from (\ref{enu:RMSISMAFR_inclusion_F}) that 
    $\RMS(X)= \RMS^\mathcal{F}(X)\subseteq \SM^\mathcal{F}(X)= \SM(X)$. 
\end{proof}

\section{Characterisation via invariant probability measures on $X^2$}
\label{sec:characterisationViaMeasures}
It is well-known that upper densities and invariant probability measures are closely related. 
In this section we will apply this interplay and present characterisations of the relations $\RMS^{(\mathcal{F})}$ and $\SM^{(\mathcal{F})}$ via invariant probability measures on $X^2$. 
For this it will be convenient to use the following notations. 

Let $(X,G)$ be an action and $\mathcal{F}=(F_n)_{n\in \mathbb{N}}$ a F{\o}lner sequence in $G$. 
For a sequence $\seq{x}=(x_n)_{n\in \mathbb{N}}$ in $X$ we denote 
$\mathcal{M}_\mathcal{F}(\seq{x})$ 
for the set of all cluster points of $((F_n)_*\delta_{x_n})_{n\in \mathbb{N}}$.
For $x\in X$ we denote 
$\mathcal{M}_\mathcal{F}(x):=\mathcal{M}_\mathcal{F}((x)_{n\in \mathbb{N}})$ 
for the set of all cluster points of $((F_n)_*\delta_{x})_{n\in \mathbb{N}}$.
Furthermore, we abbreviate 
$\mathcal{M}_G(x):=\mathcal{M}_G(\overline{G.x})$ 
for the set of all measures supported on the orbit closure of $x$. 
Note that we have 
$\mathcal{M}_{\mathcal{F}'}(x)
\subseteq\mathcal{M}_\mathcal{F}(x)
\subseteq \mathcal{M}_G(x)$ 
for any subsequence $\mathcal{F}'$ of $\mathcal{F}$. 
For the proof of the characterisation we will need the following. 

\begin{lemma}
\label{lem:densitiesAndMeasures}
    Let $(X,G)$ be an action and $\mathcal{F}$ be a F{\o}lner sequence in $G$. 
    Let $x\in X$ and $\seq{x}=(x_n)_{n \in \mathbb{N}}$ be a sequence in $X$. 
    \begin{enumerate}
    \item \label{enu:DAMopen}
    Let $U\subseteq X$ be open.
        \begin{enumerate}
        \item \label{enu:DAMopenAsymptoticDiagonal}
            For any 
            $\mu\in \mathcal{M}_\mathcal{F}(\seq{x})$ 
            we have
            $\limsup_{n\to \infty} {\haar{G_U(x_n)\cap F_n}}/{\haar{F_n}}\geq \mu(U).$
            
        \item \label{enu:DAMopenAsymptotic}
            For any 
            $\mu\in \mathcal{M}_\mathcal{F}(x)$ 
            we have
            $\upperAdens_\mathcal{F}(G_U(x))\geq \mu(U).$
            
        \item \label{enu:DAMopenBanach}
            For any 
            $\mu\in \mathcal{M}_G(x)$ 
            we have
            $\upperBdens(G_U(x))\geq \mu(U)$.
        \end{enumerate}
    \item \label{enu:DAMclosed}
    Let $A\subseteq X$ be closed. 
        \begin{enumerate}
        \item \label{enu:DAMclosedAsymptoticDiagonal}
            There exists 
            $\mu\in \mathcal{M}_\mathcal{F}(\seq{x})$ 
            with 
            $\mu(A)\geq \limsup_{n\to \infty} {\haar{G_A(x_n)\cap F_n}}/{\haar{F_n}}$.
    
        \item \label{enu:DAMclosedAsymptotic}
            There exists 
            $\mu\in \mathcal{M}_\mathcal{F}(x)$ 
            with 
            $\mu(A)\geq \upperAdens_\mathcal{F}(G_A(x))$.
        \item \label{enu:DAMclosedBanach}
            There exists 
            $\mu\in \mathcal{M}_G(x)$ 
            with 
            $\mu(A)\geq \upperBdens(G_A(x)).$
        \end{enumerate}
    \end{enumerate}
\end{lemma}
\begin{proof}
(\ref{enu:DAMopenAsymptoticDiagonal}): 
    Consider $\mu\in \mathcal{M}_\mathcal{F}(\seq{x})$. 
    Let $(F_{n_k})_{k\in \mathbb{N}}$ be a subsequence of $\mathcal{F}$ such that 
    \[(F_{n_k})_*\delta_{x_{n_k}}\to \mu.\] 
    For $k\in \mathbb{N}$ we observe that 
    \begin{align*}
        \frac{\haar{G_U(x_{n_k})\cap F_{n_k}}}{\haar{F_{n_k}}} 
        = \frac{\haar{\{g\in F_{n_k};\, g.x_{n_k}\in U\}}}{\haar{F_{n_k}}} 
        = (F_{n_k})_*\delta_{x_{n_k}}(U). 
    \end{align*}    
    Since $U$ is open and $(F_{n_k})_*\delta_{x_{n_k}}\to \mu$ it follows from the Portmanteau theorem that
    \begin{align*}
        \limsup_{n\to \infty} \frac{\haar{G_U(x_{n})\cap F_{n}}}{\haar{F_{n}}}
        &\geq \liminf_{k\to \infty}(F_{n_k})_*\delta_{x_{n_k}}(U)\geq \mu(U). 
    \end{align*}

(\ref{enu:DAMopenAsymptotic}):
    For $x\in X$ consider the constant sequence $\seq{x}:=(x)_{n\in \mathbb{N}}$
    and note that $\mathcal{M}_\mathcal{F}(\seq{x})=\mathcal{M}_\mathcal{F}(x)$. 
    Since 
    $
        \upperAdens_\mathcal{F}(G_U(x))
        =\limsup_{n\to \infty} {\haar{G_U(x)\cap F_n}}/{\haar{F_n}}
    $ 
    the statement follows from (\ref{enu:DAMopenAsymptoticDiagonal}). 
        
(\ref{enu:DAMopenBanach}): 
    Let $\mu\in \mathcal{M}_G(x)$. 
    Considering the ergodic decomposition\footnote{
    For details on the ergodic decomposition see \cite[Theorem 1.1]{GernotKlaus2000} or \cite[Theorem 8.20]{EinsiedlerWardbook2011}. 
    } 
    of $\mu$ we assume w.l.o.g.\ that $\mu$ is ergodic. 
    From Lemma \ref{lem:genericPoints} we know that there exists a F{\o}lner sequence $\mathcal{F}$ such that $x$ is $\mathcal{F}$-generic for $\mu$.
    We thus have $\mu\in \mathcal{M}_\mathcal{F}(x)$ and observe from (\ref{enu:DAMopenAsymptotic}) that 
    \begin{align*}
        \upperBdens(G_U(x))\geq \upperAdens_\mathcal{F}(G_U(x))\geq \mu(U). 
    \end{align*}

(\ref{enu:DAMclosedAsymptoticDiagonal}):
    Let $\mathcal{F}'=(F_{n_k})_{k\in \mathbb{N}}$ be a subsequence of $\mathcal{F}$ such that 
    \[\limsup_{n\to \infty} \frac{\haar{G_A(x_n)\cap F_n}}{\haar{F_n}}=\lim_{k\to \infty} \frac{\haar{G_A(x_{n_k})\cap F_{n_k}}}{\haar{F_{n_k}}}.\]
    Let $\mu$ be a cluster point of $((F_{n_k})_*\delta_{x_{n_k}})_{k\in \mathbb{N}}$ and note that 
    $\mu\in \mathcal{M}_\mathcal{F}(\seq{x})$. 
    By possibly restricting to a further subsequence we assume w.l.o.g.\ that 
    $(F_{n_k})_*\delta_{x_{n_k}}\to \mu$. 
    Since $A$ is closed we observe from the Portmanteau theorem that
     \begin{align*}
        \mu(A)
        \geq \limsup_{k\to \infty} (F_{n_k})_*\delta_{x_{n_k}}(A)
        = \lim_{k\to \infty} \frac{\haar{G_A(x_{n_k})\cap F_{n_k}}}{\haar{F_{n_k}}}
        = \limsup_{n\to \infty} \frac{\haar{G_A(x_n)\cap F_n}}{\haar{F_n}}. 
    \end{align*} 

(\ref{enu:DAMclosedAsymptotic}):
    Considering the sequence $\seq{x}:=(x)_{n\in \mathbb{N}}$ we obtain the statement from (\ref{enu:DAMclosedAsymptoticDiagonal}). 
    
(\ref{enu:DAMclosedBanach}): 
    From Lemma \ref{lem:upperBanachDensityCharacterisation} we know that there exists a F{\o}lner sequence $\mathcal{F}$ that satisfies
    $
        \upperBdens(G_A(x))
        =\upperAdens_\mathcal{F}(G_A(x))
    $
    and (\ref{enu:DAMclosedAsymptotic}) yields the existence of
    $\mu\in \mathcal{M}_\mathcal{F}(x)\subseteq \mathcal{M}_G(x)$ with 
    $\mu(A)\geq \upperAdens_\mathcal{F}(G_A(x))=\upperBdens(G_A(x))$. 
\end{proof}

We next present characterisations of
$\RMS^{(\mathcal{F})}$ and $\SM^{(\mathcal{F})}$
in terms of invariant probability measures on $X^2$.

\begin{proposition}
\label{pro:characterisationViaMeasures}
Let $(X,G)$ be an action, $(x,x')\in X^2$ and $\mathcal{F}$ be a F{\o}lner sequence in $G$.
\begin{enumerate}
    \item \label{enu:CVM_RMPF}
        The following statements are equivalent. 
        \begin{enumerate}
            \item \label{enu:CVM_RMPF_RMS}
            $(x,x')\in \RMS^\mathcal{F}(X)$.
            \item \label{enu:CVM_RMPF_limsup}
            For any open neighbourhood $U$ of $(x,x')$ 
            there exist an asymptotically diagonal sequence $(x_n,x_n')_{n \in \mathbb{N}}$ 
            and 
            $\mu_n\in \mathcal{M}_\mathcal{F}(x_n,x_n')$ for all $n\in \mathbb{N}$
            such that 
            $\limsup_{n\to \infty}\mu_n(U)>0.$
            \item \label{enu:CVM_RMPF_inf}
            For any open neighbourhood $U$ of $(x,x')$ 
            there exist an asymptotically diagonal sequence $(x_n,x_n')_{n \in \mathbb{N}}$ 
            and 
            $\mu_n\in \mathcal{M}_\mathcal{F}(x_n,x_n')$ for all $n\in \mathbb{N}$
            such that 
            $\inf_{n\in \mathbb{N}}\mu_n(U)>0.$
        \end{enumerate}    
        
    \item \label{enu:CVM_RMP}
        The following statements are equivalent. 
        \begin{enumerate}
            \item \label{enu:CVM_RMP_RMS}
            $(x,x')\in \RMS(X)$.
            \item \label{enu:CVM_RMP_limsup}
            For any open neighbourhood $U$ of $(x,x')$ 
            there exist an asymptotically diagonal sequence $(x_n,x_n')_{n \in \mathbb{N}}$ 
            and 
            $\mu_n\in \mathcal{M}_G(x_n,x_n')$ for all $n\in \mathbb{N}$
            such that 
            $\limsup_{n\to \infty}\mu_n(U)>0.$
            \item \label{enu:CVM_RMP_inf}
            For any open neighbourhood $U$ of $(x,x')$ 
            there exist an asymptotically diagonal sequence $(x_n,x_n')_{n \in \mathbb{N}}$ 
            and 
            $\mu_n\in \mathcal{M}_G(x_n,x_n')$ for all $n\in \mathbb{N}$
            such that 
            $\inf_{n\in \mathbb{N}}\mu_n(U)>0.$
        \end{enumerate}    
        
    \item \label{enu:CVM_SMF}
        We have $(x,x')\in \SM^\mathcal{F}(X)$
            if and only if 
        for any open neighbourhood $U$ of $(x,x')$
        there exist an asymptotically diagonal sequence $\seq{x}=(x_n,x_n')_{n \in \mathbb{N}}$ 
        and $\mu\in \mathcal{M}_\mathcal{F}(\seq{x})$ with $\mu(U)>0$.

    \item \label{enu:CVM_SM}
        We have $(x,x')\in \SM(X)$
            if and only if 
        for any open neighbourhood $U$ of $(x,x')$
        there exist an asymptotically diagonal sequence $\seq{x}=(x_n,x_n')_{n \in \mathbb{N}}$, a F{\o}lner sequence $\mathcal{F}$ 
        and $\mu\in \mathcal{M}_\mathcal{F}(\seq{x})$ with $\mu(U)>0$.
\end{enumerate}
\end{proposition}
\begin{proof}
(\ref{enu:CVM_RMPF_RMS})$\Rightarrow$(\ref{enu:CVM_RMPF_inf}):
    Consider $(x,x')\in \RMS^\mathcal{F}(X)$ and let $U$ be an open neighbourhood of $(x,x')$. 
    Choose a closed neighbourhood $A$ and an open neighbourhood $V$ of $(x,x')$ such that 
    $V\subseteq A\subseteq U$. 
    From Proposition \ref{pro:infimumReformulation} we know that there exists an asymptotically diagonal sequence $(x_n,x_n')_{n\in \mathbb{N}}$ in $X^2$ with 
    \[c:=\inf_{n\in \mathbb{N}} \upperAdens_\mathcal{F}(G_V(x_n,x_n'))>0.\]
    For $n\in \mathbb{N}$ Lemma \ref{lem:densitiesAndMeasures} allows to choose $\mu_n\in \mathcal{M}_\mathcal{F}(x_n,x_n')$ such that 
    \begin{align*}
    \mu_n(U)
    \geq\mu_n(A)
    \geq \upperAdens_\mathcal{F}(G_A(x_n,x_n'))
    \geq \upperAdens_\mathcal{F}(G_V(x_n,x_n'))\geq c>0.
    \end{align*}
    
(\ref{enu:CVM_RMPF_inf})$\Rightarrow$(\ref{enu:CVM_RMPF_limsup}):
    Trivial. 

(\ref{enu:CVM_RMPF_limsup})$\Rightarrow$(\ref{enu:CVM_RMPF_RMS}):
    Consider an open neighbourhood $U$ of $(x,x')$. 
    By our assumption on $(x,x')$ there exists an asymptotically diagonal sequence $(x_n,x_n')_{n\in \mathbb{N}}$ in $X^2$ such that for all $n\in \mathbb{N}$
    there exists $\mu_n\in \mathcal{M}_\mathcal{F}(x_n,x_n')$ satisfying $\limsup_{n\to \infty}\mu_n(U)>0$. 
    From Lemma \ref{lem:densitiesAndMeasures} we observe 
    \begin{align*}
        \limsup_{n\to \infty} \upperAdens_\mathcal{F}(G_U(x_{n},x_{n}'))
        \geq \limsup_{n\to \infty} \mu_{n}(U) >0.
    \end{align*}
    This shows $(x,x')\in \RMS^\mathcal{F}(X)$. 
    
(\ref{enu:CVM_RMP}):
    This follows from a similar argument as (\ref{enu:CVM_RMPF}) applying the respective statements about upper Banach densities from Lemma \ref{lem:densitiesAndMeasures}. 

(\ref{enu:CVM_SMF}, '$\Rightarrow$'):
    Consider a pair $(x,x')\in \SM^\mathcal{F}(X)$ and let $U$ be an open neighbourhood of $(x,x')$. 
    Choose a closed neighbourhood $A$ and an open neighbourhood $V$ of $(x,x')$ such that 
    $V\subseteq A\subseteq U$. 
    Since $(x,x')\in \SM^\mathcal{F}(X)$ there exists an asymptotically diagonal sequence $\seq{x}=(x_n,x_n')_{n\in \mathbb{N}}$ in $X^2$ with
    \[\limsup_{n\to \infty} \frac{\haar{G_V(x_n,x_n')\cap F_n}}{\haar{F_n}}>0.\]
    It follows from Lemma \ref{lem:densitiesAndMeasures} that there exists $\mu\in \mathcal{M}_\mathcal{F}(\seq{x})$ with 
    \[
    \mu(U)
    \geq \mu(A)
    \geq \limsup_{n\to \infty} \frac{\haar{G_A(x_n,x_n')\cap F_n}}{\haar{F_n}}
    \geq \limsup_{n\to \infty} \frac{\haar{G_V(x_n,x_n')\cap F_n}}{\haar{F_n}}
    >0.
    \]

(\ref{enu:CVM_SMF}, '$\Leftarrow$'):
    Consider an open neighbourhood $U$ of $(x,x')$. 
    By our assumption on $(x,x')$ there exists an asymptotically diagonal sequence $\seq{x}=(x_n,x_n')_{n\in \mathbb{N}}$ in $X^2$ and $\mu\in \mathcal{M}_\mathcal{F}(\seq{x})$ with $\mu(U)>0$.  
    From Lemma \ref{lem:densitiesAndMeasures} we observe 
    \begin{align*}
        \limsup_{n\to\infty} \frac{\haar{G_U(x_n,x_n')\cap F_n}}{\haar{F_n}}
        \geq \mu(U)>0.
    \end{align*}
    This shows $(x,x')\in \SM^\mathcal{F}(X)$. 

(\ref{enu:CVM_SM}):
    This follows with a similar argument as for (\ref{enu:CVM_SMF}) from Lemma \ref{lem:densitiesAndMeasures}. 
\end{proof}

\section{The maximal $(\mathcal{F}-)$mean equicontinuous factor}
\label{sec:mainTheorem}
This section is devoted to the proof of our main result, the characterisations of the maximal $(\mathcal{F}-)$mean equicontinuous factor via the introduced relations. 

\subsection{The maximal $(\mathcal{F}-)$mean equicontinuous factor via $\RMS^{(\mathcal{F})}$}
The following lemma records the stability of regional (joint) mean sensitivity under factor maps.

\begin{lemma}
\label{lem:factorStabilityI}
    Let $\pi\colon X\to Y$ be a factor map and let $\mathcal{F}$ be a F{\o}lner sequence. 
    For any $(x,x')\in \RMS^\mathcal{F}(X)$ we have $(\pi(x),\pi(x'))\in \RMS^\mathcal{F}(Y)$. 
    A similar statement holds for 
    $\RMS$, 
    $\SM^\mathcal{F}$
    and $\SM$. 
\end{lemma}
\begin{proof}
    Denote $\phi\colon X^2\to Y^2$ for the factor map $(x,x')\mapsto (\pi(x),\pi(x'))$. 
    For any open set $V\subseteq Y^2$ we have that $U:=\phi^{-1}(V)$ satisfies $G_U(z,z')\subseteq G_V(\pi(z),\pi(z'))$ for all $(z,z')\in X^2$.
    From this it is straightforward to observe that whenever $U$ is $\mathcal{F}$-diagonally linked, then also $V$ must be $\mathcal{F}$-diagonally linked and we observe the statement about $\RMS^\mathcal{F}$. 
    The statements about $\RMS$ and for $\SM^\mathcal{F}$ and $\SM$ can be shown with a similar argument. 
\end{proof}

\begin{lemma}
\label{lem:supportControlForFactorStability}
    Let $(X,G)$ be an action, $\mathcal{F}$ be a F{\o}lner sequence, and $(x_n,x_n')_{n\in \mathbb{N}}$ be an asymptotically diagonal sequence. 
    For $n\in \mathbb{N}$ consider    
    $\mu_n\in \mathcal{M}_\mathcal{F}(x_n,x_n')$. 
    For any cluster point $\mu$ of $(\mu_n)_{n\in \mathbb{N}}$ we have $\supp(\mu)\subseteq \RMS^\mathcal{F}(X)$. 
\end{lemma}
\begin{proof}
    Consider $(x,x')\in \supp(\mu)$ and let $U$ be an open neighbourhood of $(x,x')$. 
    Note that $\mu(U)>0$. 
    Consider a strictly increasing sequence $(n_k)_{k\in \mathbb{N}}$ in $\mathbb{N}$ such that $\mu_{n_k} \to \mu$. 
    From the Portmanteau theorem we observe that 
    \begin{align*}
        \limsup_{n\to \infty}\mu_n(U)
        \geq \liminf_{k\to \infty}\mu_{n_k}(U)
        \geq \mu(U)>0
    \end{align*}
    and conclude $(x,x')\in \RMS^\mathcal{F}(X)$ from Proposition \ref{pro:characterisationViaMeasures}.
\end{proof}

\begin{lemma}
\label{lem:factorStabilityII}
    Let $\pi\colon X\to Y$ be a factor map and let $\mathcal{F}$ be a F{\o}lner sequence.
    \begin{enumerate}
        \item \label{enu:FSIIasymptotical}
            If $\RMS^\mathcal{F}(X)\subseteq R(\pi)$, then $(Y,G)$ is $\mathcal{F}$-mean equicontinuous.
        \item \label{enu:FSIIBanach}
            If $\RMS(X)\subseteq R(\pi)$, then $(Y,G)$ is mean equicontinuous.
    \end{enumerate}
\end{lemma}
\begin{proof}
(\ref{enu:FSIIasymptotical}):
    If $Y$ is not $\mathcal{F}$-mean equicontinuous there exists $y\in Y$ that is not $\mathcal{F}$-mean equicontinuous. 
    For such $y$ we find $\delta>0$ and a sequence $(y_n)_{n\in \mathbb{N}}$ in $Y$ with $y_n\to y$ such that
    $D_\mathcal{F}(y_n,y)\overset{n}{\to} \delta$ and 
    $D_\mathcal{F}(y_n,y)\geq \delta$ for all $n\in \mathbb{N}$. 

    For $n\in \mathbb{N}$ we choose $x_n\in X$ with $\pi(x_n)=y_n$ and a subsequence $\mathcal{F}^{(n)}=(F_k^{(n)})_{k\in \mathbb{N}}$ of $\mathcal{F}$ such that 
    \[\lim_{k\to \infty} (F_k^{(n)})^*d(y_n,y)=D_\mathcal{F}(y_n,y).\] 
    We denote
    $\mu_k^{(n)}:=(F_k^{(n)})_*\delta_{(x_n,x)}$. 
    By possibly passing to a further subsequence of $\mathcal{F}^{(n)}$ we assume w.l.o.g.\ that
    there exists $\mu^{(n)}\in \mathcal{M}_G(X^2)$ with $\mu_k^{(n)}\overset{k}{\to} \mu^{(n)}$.
    Furthermore, by possibly considering a subsequence of $(y_n)_{n\in \mathbb{N}}$ we assume w.l.o.g.\ that $x_n\overset{n}{\to} x$ in $X$, and that $\mu^{(n)}\overset{n}{\to} \mu$ in $\mathcal{M}_G(X^2)$. 
    Note that we have $\pi(x)=y$. 

    Denote by $\phi\colon X^2\to Y^2$ the factor map given by $(x_1,x_2)\mapsto (\pi(x_1),\pi(x_2))$.
    For $n,k\in \mathbb{N}$ we denote  
    $\nu:=\phi_*\mu$, 
    $\nu^{(n)}:=\phi_*\mu^{(n)}$
    and 
    $\nu_k^{(n)}:=\phi_*\mu_k^{(n)}$.
    Since $\phi_*\colon \mathcal{M}(X^2)\to \mathcal{M}(Y^2)$ is affine and continuous it follows that
    \begin{align*}
        (F_k^{(n)})_*\delta_{(y_n,y)}
        &=(F_k^{(n)})_*\phi_*\delta_{(x_n,x)}
        =\phi_*(F_k^{(n)})_*\delta_{(x_n,x)}
        =\nu_k^{(n)}\overset{k}{\to} \nu^{(n)}\overset{n}{\to} \nu.
    \end{align*}
    Furthermore, since $(x_n,x)_{n\in \mathbb{N}}$ is an asymptotically diagonal sequence with 
    \begin{align*}
        \mu^{(n)}\in \mathcal{M}_{\mathcal{F}^{(n)}}(x_n,x)\subseteq \mathcal{M}_{\mathcal{F}}(x_n,x)
    \end{align*}
    for all $n\in \mathbb{N}$, we observe from $\mu^{(n)}\overset{n}{\to} \mu$ and Lemma \ref{lem:supportControlForFactorStability} that 
    \begin{align*}
        \supp(\mu)\subseteq \RMS^{\mathcal{F}}(X)\subseteq R(\pi).  
    \end{align*}
    We thus have 
    \begin{align*}
        \supp(\nu)=\phi(\supp(\mu))\subseteq \phi(R(\pi))=\Delta_Y.
    \end{align*}
    Since $d\equiv 0$ on $\Delta_Y$ we have the following contradiction
    \begin{align*}
        0
        &=\nu(d)
        =\lim_{n\to \infty} \nu^{(n)}(d)
        =\lim_{n\to \infty} \lim_{k\to \infty} \nu_k^{(n)}(d)
        =\lim_{n\to \infty} \lim_{k\to \infty} (F_k^{(n)})_*\delta_{(y_n,y)}(d)\\
        &=\lim_{n\to \infty} \lim_{k\to \infty} (F_k^{(n)})^*d(y_n,y)
        =\lim_{n\to \infty} D_\mathcal{F}(y_n,y)=\delta>0.
    \end{align*}        

(\ref{enu:FSIIBanach}):
    Given a F{\o}lner sequence $\mathcal{F}$ we have  
    $\RMS^\mathcal{F}(X)\subseteq \RMS(X)\subseteq R(\pi)$ and (\ref{enu:FSIIasymptotical}) yields that 
    $(X,G)$ is $\mathcal{F}$-mean equicontinuous. 
    This shows that $(X,G)$ is $\mathcal{F}$-mean equicontinuous for all F{\o}lner sequences $\mathcal{F}$ and Lemma \ref{lem:meanEquicontinuityCharacterisation} implies that $(X,G)$ is mean equicontinuous. 
\end{proof}

\begin{lemma}
\label{lem:DFviaMeasures}
    Let $(X,G)$ be an action and $\mathcal{F}$ be a F{\o}lner sequence. For any $(x,x')\in X^2$ we have 
    $D_\mathcal{F}(x,x')=\sup_{\mu\in \mathcal{M}_\mathcal{F}(x,x')}\mu(d).$
\end{lemma}
\begin{proof}
    Denoting $\mathcal{F}=(F_n)_{n\in \mathbb{N}}$ we observe 
    \begin{align*}
        D_\mathcal{F}(x,x')
        &= \limsup_{n\to \infty} F_n^*d(x,x')
        =\limsup_{n\to \infty} (F_n)_*\delta_{(x,x')}(d)=\sup_{\mu\in \mathcal{M}_\mathcal{F}(x,x')}\mu(d).
        \qedhere
    \end{align*}
\end{proof}

\begin{proposition}
\label{pro:meanEquicontinuityViaRMS}
    Let $(X,G)$ be an action and $\mathcal{F}=(F_n)$ be a F{\o}lner sequence.
    \begin{enumerate}
        \item \label{enu:MEVRMS_F}
        $(X,G)$ is $\mathcal{F}$-mean equicontinuous if and only if $\RMS^\mathcal{F}(X)\subseteq \Delta_X$. 
        \item \label{enu:MEVRMS}
        $(X,G)$ is mean equicontinuous if and only if $\RMS(X)\subseteq \Delta_X$. 
    \end{enumerate}
\end{proposition}
\begin{proof}
(\ref{enu:MEVRMS_F}):
    Assume that $\RMS^\mathcal{F}(X)\subseteq\Delta_X$. 
    Consider the trivial factor map $\pi\colon X\to X$ with $x\mapsto x$. 
    We have $\RMS^\mathcal{F}(X)\subseteq \Delta_X=R(\pi)$ and observe $(X,G)$ to be $\mathcal{F}$-mean equicontinuous from Lemma \ref{lem:factorStabilityII}. 

    For the converse, consider $(x,x')\in \RMS^\mathcal{F}(X)$ and assume that $x\neq x'$. 
    There exists $\epsilon>0$ such that $U:=d^{-1}((\epsilon,\infty))$ is an open neighbourhood of $(x,x')$. 
    Since $(x,x')\in \RMS^\mathcal{F}(X)$ we know from Proposition \ref{pro:characterisationViaMeasures} that there exists an asymptotically diagonal sequence $(x_n,x_n')_{n\in \mathbb{N}}$ and for each $n\in \mathbb{N}$ some $\mu_n\in \mathcal{M}_\mathcal{F}(x_n,x_n')$ such that 
    $c:=\inf_{n\in \mathbb{N}} \mu_n(U)>0.$
    For each $n\in \mathbb{N}$ we thus observe from Lemma \ref{lem:DFviaMeasures} that  
    \begin{align*}
        D_\mathcal{F}(x_n,x_n')=\sup_{\mu\in \mathcal{M}_\mathcal{F}(x_n,x_n')}\mu(d)
        \geq \mu_n(d)
        \geq \epsilon\mu_n(U)
        \geq c\epsilon>0. 
    \end{align*}
    However, since $(X,G)$ is $\mathcal{F}$-mean equicontinuous, the condition $d(x_n,x_n')\to 0$ implies $D_\mathcal{F}(x_n,x_n')\to 0$, a contradiction.

(\ref{enu:MEVRMS}):
    Recall that $(X,G)$ is mean equicontinuous if and only if it is $\mathcal{F}$-mean equicontinuous for all F{\o}lner sequences $\mathcal{F}$.
    Furthermore, from Theorem \ref{the:RMSInSMandFolnerRealisation} we know that $\RMS(X)=\RMS^\mathcal{F}(X)$ for some F{\o}lner sequence $\mathcal{F}$.
    We thus conclude (\ref{enu:MEVRMS}) from (\ref{enu:MEVRMS_F}). 
\end{proof}

\begin{theorem}
\label{the:meanEquicontinuityFactorCharacterisation}
    Let $\pi\colon X\to Y$ be a factor map and $\mathcal{F}$ be a F{\o}lner sequence. 
    \begin{enumerate}
        \item \label{enu:MEC_factor_F}
        $(Y,G)$ is $\mathcal{F}$-mean equicontinuous if and only if $\RMS^\mathcal{F}(X)\subseteq R(\pi)$. 

        \item \label{enu:MEC_factor}
        $(Y,G)$ is mean equicontinuous if and only if $\RMS(X)\subseteq R(\pi)$. 
    \end{enumerate}
\end{theorem}
\begin{proof}
(\ref{enu:MEC_factor_F}): 
    From Lemma \ref{lem:factorStabilityII} we already know that whenever $\RMS^\mathcal{F}(X)\subseteq R(\pi)$, then $(Y,G)$ is $\mathcal{F}$-mean equicontinuous. 
    For the converse assume that $(Y,G)$ is $\mathcal{F}$-mean equicontinuous and consider $(x,x')\in \RMS^\mathcal{F}(X)$. 
    From Lemma \ref{lem:factorStabilityI} and Proposition \ref{pro:meanEquicontinuityViaRMS} we observe that 
    $(\pi(x),\pi(x'))\in \RMS^\mathcal{F}(Y)\subseteq\Delta_Y$ and hence $(x,x')\in R(\pi)$. 
    This shows $\RMS^\mathcal{F}(X)\subseteq R(\pi)$.     
    
(\ref{enu:MEC_factor}): 
    This follows from a similar argument. 
\end{proof}

\begin{theorem}    
\label{the:maximalMeanEquicontinuousFactor}
    Let $(X,G)$ be an action and let $\mathcal{F}$ be a F{\o}lner sequence.
    \begin{enumerate}
        \item \label{enu:MMEF_F}
        $(X\big/\icerhull{\RMS^\mathcal{F}(X)}, G)$ is the maximal $\mathcal{F}$-mean equicontinuous factor of $(X,G)$.
        
        \item \label{enu:MMEF}
        $(X\big/\icerhull{\RMS(X)}, G)$ is the maximal mean equicontinuous factor of $(X,G)$.  
    \end{enumerate}
\end{theorem}
\begin{proof}
(\ref{enu:MMEF_F}): 
    By Theorem \ref{the:meanEquicontinuityFactorCharacterisation},  
    $(X\big/\icerhull{\RMS^\mathcal{F}(X)}, G)$ is mean equicontinuous. 
    Furthermore, for any mean equicontinuous factor $(Y,G)$ of $(X,G)$ given by a factor map $\pi\colon X\to Y$ we observe from Theorem \ref{the:meanEquicontinuityFactorCharacterisation} that $\RMS^\mathcal{F}(X)\subseteq R(\pi)$. 
    Since $R(\pi)$ is a closed invariant equivalence relation we have 
    $\icerhull{ \RMS^\mathcal{F}(X)}\subseteq R(\pi)$ 
    and it follows that $(Y,G)$ is a factor of 
    $(X\big/\icerhull{\RMS^\mathcal{F}(X)}, G)$. 

(\ref{enu:MMEF}):
    This follows from a similar argument. 
\end{proof}

\subsection{The maximal mean equicontinuous factor via $\SM$ and $\SM^{\mathcal{F}}$}

As presented in the introduction the quotient 
$X\big/ \icerhull{\SM^{\mathcal{F}}(X)}$ 
does not always give the maximal $\mathcal{F}$-mean equicontinuous factor.
We next show that in contrast the relation $\SM(X)$ can always be used to describe the maximal mean equicontinuous factor. 
Furthermore, we provide additional assumptions under which $\SM^{\mathcal{F}}(X)$ yields a description of the maximal $\mathcal{F}$-mean equicontinuous factor.

For this we will need the following well-known result. 
It is a direct consequence of
\cite[Theorem~1.2]{fuhrmann2022structure} and
\cite[Theorem~3.12]{fuhrmann2025continuity}.
Related statements can be found in
\cite[Theorem~4.9]{downarowicz2020when} and
\cite[Theorem~1.6]{xu2024weak}.

\begin{lemma}
\label{lem:uniformity}
    Whenever $(X,G)$ is mean equicontinuous, then for any F{\o}lner sequence $\mathcal{F}=(F_n)_{n\in \mathbb{N}}$ we have $F_n^*d\to D_\mathcal{F}$ in $C(X^2)$ as $n\to\infty$. 
\end{lemma}

We apply this lemma to obtain the following. 

\begin{lemma}
\label{lem:MMEFviaSMpreparations}
    Let $(X,G)$ be a mean equicontinuous action. 
  Then $\SM(X)\subseteq \Delta_X$. 
\end{lemma}
\begin{proof}
    Consider $(x,x')\in \SM(X)\setminus \Delta_X$. 
    There exists $\epsilon>0$ such that $U:=d^{-1}((\epsilon,\infty))$ is an open neighbourhood of $(x,x')$ and 
    Proposition \ref{pro:infimumReformulation} allows to choose a F{\o}lner sequence $\mathcal{F}=(F_n)_{n\in \mathbb{N}}$ and an asymptotically diagonal sequence $(x_n,x_n')_{n\in \mathbb{N}}$ such that 
    \begin{align*}
        c:=\inf_{n\in\mathbb{N}} \frac{\haar{G_U(x_n,x_n')\cap F_n}}{\haar{F_n}}>0. 
    \end{align*}
    For $g\in G_U(x_n,x_n')\cap F_n$ we know $(g.x_n,g.x_n')\in U$ and hence $d(g.x_n,g.x_n')>\epsilon$. 
    Thus, for $n\in\mathbb{N}$ we have
    \begin{align*}
        F_n^*d(x_n,x_n')
        \geq \frac{1}{\haar{F_n}}\int_{ G_U(x_n,x_n')\cap F_n} d(g.x_n,g.x_n') \dd g
        \geq c\epsilon>0.
    \end{align*}
    Since $(X,G)$ is mean equicontinuous, we observe from Lemma \ref{lem:uniformity} that 
    $F_n^*d\to D_\mathcal{F}$ in $C(X^2)$. 
    Furthermore, from $d(x_n,x_n')\to 0$ we know $D_\mathcal{F}(x_n,x_n')\to 0$ and hence there exists $n\in \mathbb{N}$ such that 
    \[\|F_m^*d-D_\mathcal{F}\|_\infty+D_\mathcal{F}(x_m,x_m')< c\epsilon\]
    holds for all $m\geq n$.
    In particular, we have
    \begin{align*}
        F_n^*d(x_n,x_n')
        \leq |F_n^*d(x_n,x_n')-D_\mathcal{F}(x_n,x_n')| + 
            D_\mathcal{F}(x_n,x_n')
        <c\epsilon,
    \end{align*}
    a contradiction.     
\end{proof}

\begin{theorem}
\label{the:theMMEFviaSM}
     Let $(X,G)$ be an action and $\pi\colon X\to Y$ a factor map. 
     \begin{enumerate}
         \item \label{enu:MMEFviaSM_action}
         $(X,G)$ is mean equicontinuous if and only if $\SM(X)\subseteq \Delta_X$. 
         \item \label{enu:MMEFviaSM_factor}
         $(Y,G)$ is mean equicontinuous if and only if $\SM(X)\subseteq R(\pi)$. 
         \item \label{enu:MMEFviaSM}
         $(X\big/ \icerhull{\SM(X)}, G)$ is the maximal mean equicontinuous factor of $(X,G)$. 
     \end{enumerate}
\end{theorem}
\begin{proof}
(\ref{enu:MMEFviaSM_action}): 
    From Lemma \ref{lem:MMEFviaSMpreparations} we know that whenever $(X,G)$ is mean equicontinuous, then $\SM(X)\subseteq \Delta_X$. 
    For the converse assume that $\SM(X)\subseteq \Delta_X$. 
    From Theorem \ref{the:RMSInSMandFolnerRealisation} we know that $\RMS(X)\subseteq \SM(X)\subseteq \Delta_X$ and Proposition \ref{pro:meanEquicontinuityViaRMS} implies that $(X,G)$ is mean equicontinuous. 
    
(\ref{enu:MMEFviaSM_factor}): 
    With a similar argument we observe from the Theorems \ref{the:RMSInSMandFolnerRealisation} and \ref{the:meanEquicontinuityFactorCharacterisation} that whenever $\SM(X)\subseteq R(\pi)$, then $(Y,G)$ is mean equicontinuous. 
    For the converse assume that $(Y,G)$ is mean equicontinuous and consider $(x,x')\in \SM(X)$. 
    It follows from Lemma \ref{lem:factorStabilityI} and (\ref{enu:MMEFviaSM_action}) that $(\pi(x),\pi(x'))\in \SM(Y)\subseteq \Delta_Y$, and hence $(x,x')\in R(\pi)$. 
    
(\ref{enu:MMEFviaSM}): 
    This follows from (\ref{enu:MMEFviaSM_factor}) with an argument similar to the one presented for Theorem \ref{the:maximalMeanEquicontinuousFactor}. 
\end{proof}

In order to develop appropriate conditions that allow to use $\SM^\mathcal{F}$ for the description of the maximal $\mathcal{F}-$mean equicontinuous action we will use the following, which can be found in \cite[Theorem 5.1 and Theorem 5.8]{fuhrmann2022structure}. 
\begin{proposition}
\label{pro:FGL_FMEvsME}
    Let $(X,G)$ be an action and let $\mathcal{F}$ be a F{\o}lner sequence. 
    Assume that at least one of the following conditions is satisfied:
    \begin{itemize}
        \item $(X,G)$ is fully supported. 
        \item $\mathcal{F}$ is a two-sided F{\o}lner sequence. 
    \end{itemize}
    Then $(X,G)$ is mean equicontinuous if and only if it is $\mathcal{F}$-mean equicontinuous. 
\end{proposition}

\begin{remark}
    Note that these conditions are trivially satisfied, whenever $G$ is Abelian or whenever $(X,G)$ is minimal. 
\end{remark}

\begin{corollary}
\label{cor:theMMEFviaSM_F}
    Let $(X,G)$ be an action, $\mathcal{F}$ be a F{\o}lner sequence in $G$ and $\pi\colon X\to Y$ be a factor map.
   Under the hypotheses of Proposition \ref{pro:FGL_FMEvsME},
    the following statements hold. 
    \begin{enumerate}
        \item \label{enu:MMEFviaSMFolner_action}
        $(X,G)$ is $\mathcal{F}-$mean equicontinuous if and only if $\SM^\mathcal{F}(X)\subseteq \Delta_X$. 
        \item \label{enu:MMEFviaSMFolner_factor}
        $(Y,G)$ is $\mathcal{F}-$mean equicontinuous if and only if $\SM^\mathcal{F}(X)\subseteq R(\pi)$. 
        \item \label{enu:MMEFviaSMFolner}
        $(X\big/\icerhull{\SM^\mathcal{F}(X)},G)$ is the maximal $\mathcal{F}-$mean equicontinuous factor of $(X,G)$. 
    \end{enumerate}    
\end{corollary}
\begin{proof}
(\ref{enu:MMEFviaSMFolner_action}): 
    Whenever $(X,G)$ is $\mathcal{F}-$mean equicontinuous, then our assumptions and  Proposition \ref{pro:FGL_FMEvsME} yield that 
    $(X,G)$ is mean equicontinuous. 
    Thus, Theorem \ref{the:theMMEFviaSM} yields that $\SM^\mathcal{F}(X)\subseteq \SM(X)\subseteq \Delta_X$. 
    For the converse assume that $\SM^\mathcal{F}(X)\subseteq \Delta_X$. 
    From Theorem \ref{the:RMSInSMandFolnerRealisation} we observe that $\RMS^\mathcal{F}(X)\subseteq \SM^\mathcal{F}(X)\subseteq \Delta_X$ and Proposition \ref{pro:meanEquicontinuityViaRMS} implies that $(X,G)$ is $\mathcal{F}$-mean equicontinuous. 

(\ref{enu:MMEFviaSMFolner_factor}):
    This follows from a similar combination of the 
    Theorems \ref{the:RMSInSMandFolnerRealisation} and  \ref{the:theMMEFviaSM} with the Propositions \ref{pro:meanEquicontinuityViaRMS} and \ref{pro:FGL_FMEvsME}.

(\ref{enu:MMEFviaSMFolner}):
     This follows from (\ref{enu:MMEFviaSMFolner_factor}) with an argument similar to the one presented for Theorem \ref{the:maximalMeanEquicontinuousFactor}. 
\end{proof}

\subsection{The interplay with the proximal relation}
Let $(X,G)$ be an action. 
A pair $(x,x')\in X^2$ is called \emph{proximal} if $\inf_{g\in G}d(g.x,g.x')=0$ \cite{auslander1988minimal}. 
We denote $\P(X)$ for the set of all proximal pairs.
Clearly, we have $\P(X)\subseteq \RP(X)$ and Proposition \ref{pro:basicPropertiesRMSandSM} and Theorem \ref{the:RMSInSMandFolnerRealisation} yield
\[\P(X)\cup \RMS(X)\subseteq \P(X)\cup \SM(X)\subseteq \RP(X).\]
An action $(X,G)$ is called \emph{distal} if $\P(X)\subseteq \Delta_X$. 
It is well-known that also the notion of distality allows for the concept of a maximal distal factor and that the latter can be described by $X\big/\icerhull{\P(X)}$ \cite{auslander1988minimal}. 
Furthermore, from \cite[Theorem 9.1]{hauser2024mean} or \cite[Corollary 3.6]{li2015mean} we know\footnote{
Note that the proof is presented there only in the context of actions of countable amenable groups but allows for a straightforward generalization to our context. 
} that an action $(X,G)$ is equicontinuous if and only if it is mean equicontinuous and distal. From this observation it is straightforward to conclude that the simultaneous consideration of proximality and regionally (joint) mean sensitivity allows to describe the maximal equicontinuous factor of an action, i.e.\ the following. 
We include the short proof for the convenience of the reader. 

\begin{corollary}\label{cor:interplayProximalRelation}
    For any action $(X,G)$ we have 
    \[\icerhull{\RP(X)}
    = \icerhull{\P(X)\cup \RMS(X)} 
    = \icerhull{\P(X)\cup \SM(X)}.\]
\end{corollary}
\begin{proof}
    Clearly, we have 
    \[\icerhull{\P(X)\cup \RMS(X)}\subseteq \icerhull{\P(X)\cup \SM(X)}\subseteq \icerhull{\RP(X)}.\]

    For the converse note that $X/\icerhull{\P(X) \cup \RMS(X)}$ is a factor of both, the maximal distal factor $X/\icerhull{\P(X)}$, and the maximal mean equicontinuous factor $X/\icerhull{\RMS(X)}$ of $(X,G)$.
    We thus know that $X/\icerhull{\P(X) \cup \RMS(X)}$ is distal and mean equicontinuous and hence equicontinuous.
    Since $X/\icerhull{\RP(X)}$ is the maximal equicontinuous factor
    we observe that $X/\icerhull{\P(X) \cup \RMS(X)}$ is a factor of $X/\icerhull{\RP(X)}$ and conclude $\icerhull{\RP(X)}\subseteq \icerhull{\P(X) \cup \RMS(X)}$. 
\end{proof}

\section{Reflexivity and the maximal support}
\label{sec:reflexivity}
In this section, we characterise when the relations $\RMS^{(\mathcal{F})}(X)$ and $\SM^{(\mathcal{F})}(X)$ are reflexive, and prove that the maximal mean equicontinuous factor map is one-to-one outside $\supp(X,G)$. 
For this we first provide some details on the interplay of the considered relations and the maximal support.  

\begin{theorem}
\label{the:maximalSupportAndRMS_SM}
    Let $(X,G)$ be an action. 
    \begin{enumerate}
        \item \label{enu:MSARMSSM}
        We have
        \[\RMS(X)\subseteq \SM(X)\subseteq \supp(X,G)^2.\] 

        \item \label{enu:MSARMSSM_diagonal}
        We have
        \[\Delta_{\supp(X,G)} 
        = \Delta_X\cap \RMS(X)
        = \Delta_X\cap \SM(X).\] 

        \item \label{enu:MSARMSSM_diagonal_F}
        For any F{\o}lner sequence $\mathcal{F}$ in $G$ we have
        \[
        \Delta_{\supp(X,G)} 
        = \Delta_X\cap \RMS^\mathcal{F}(X)
        = \Delta_X\cap \SM^\mathcal{F}(X)
        .\]
    \end{enumerate}
\end{theorem}

\begin{remark}
    Recall from the Propositions \ref{pro:interplayweaksitmAndregionallysitm} and \ref{pro:basicPropertiesRMSandSM} that 
    \[\WSM^\mathcal{F}(X)\subseteq \SM^\mathcal{F}(X)\subseteq\SM(X).\] 
    We thus observe from (\ref{enu:MSARMSSM}) that $\WSM^\mathcal{F}(X)\subseteq \supp(X,G)^2$. This observation can also be found in \cite[Proposition 5.14]{li2021mean}, where it was shown for actions of $\mathbb{Z}$ and with respect to the standard F{\o}lner sequence $\mathcal{F}=(\{0,\dots,n-1\})_{n\in \mathbb{N}}$. 
\end{remark}

\begin{proof}
(\ref{enu:MSARMSSM}):
    The first inclusion was shown in Theorem \ref{the:RMSInSMandFolnerRealisation}. 
    Let $(x,x')\in \SM(X)$ and consider an open neighbourhood $U$ of $(x,x')$. From Proposition \ref{pro:characterisationViaMeasures} we observe the existence of $\mu\in \mathcal{M}_G(X^2)$ such that $\mu(U)>0$. It follows from the Lemmas \ref{lem:characterisationMaximalSupport} and 
    \ref{lem:productSupport} that 
    $(x,x')\in \supp(X^2,G)=\supp(X,G)^2$. 
    
(\ref{enu:MSARMSSM_diagonal} \& \ref{enu:MSARMSSM_diagonal_F}):
    Recall from (\ref{enu:MSARMSSM}), Proposition \ref{pro:basicPropertiesRMSandSM} and Theorem \ref{the:RMSInSMandFolnerRealisation} that for a F{\o}lner sequence $\mathcal{F}$ we have 
    $\RMS^\mathcal{F}(X)\subseteq \RMS(X)\subseteq \SM(X)\subseteq \supp(X,G)^2$ 
    and 
    $\RMS^\mathcal{F}(X)\subseteq \SM^\mathcal{F}(X)\subseteq \SM(X)\subseteq \supp(X,G)^2$.
    It thus remains to show that $\Delta_{\supp(X,G)}\subseteq \RMS^\mathcal{F}(X)$. 

    For this consider $x\in \supp(X,G)$ and let $U$ be an open neighbourhood of $(x,x)$ in $X^2$. 
    Consider $V:=U\cap \Delta_X$. 
    Identifying $\Delta_X$ with $X$ we observe $V$ to be an open neighbourhood of $x$. 
    Since $x\in \supp(X,G)$ we obtain from Lemma \ref{lem:characterisationMaximalSupport} that there exists some $\mu\in \mathcal{M}_G(X)$ such that $\mu(V)>0$. 
    In consideration of the ergodic decomposition of $\mu$ we assume w.l.o.g.\ that $\mu$ is ergodic.
    By Lemma \ref{lem:genericPoints} there exists $x'\in X$ and a subsequence $\mathcal{F}'$ of $\mathcal{F}$ such that $x'$ is $\mathcal{F}'$-generic for $\mu$. 
    In particular, we have $\mu\in \mathcal{M}_{\mathcal{F}'}(x')\subseteq \mathcal{M}_\mathcal{F}(x')$. 
    Interpreting $X$ again as $\Delta_X$ we have found $(x',x')\in \Delta_X$ and $\mu\in \mathcal{M}_\mathcal{F}(x',x')$ with $\mu(U)=\mu(V)>0$. 
    
    Considering the asymptotically diagonal sequence $(x_n,x_n'):=(x',x')$ we 
    have $\mu\in \mathcal{M}_\mathcal{F}(x_n,x_n')$ for all $n\in \mathbb{N}$. 
    Thus, we observe from Proposition \ref{pro:characterisationViaMeasures} that
    $(x,x)\in \RMS^\mathcal{F}(X)$.
\end{proof}

Recall that an action is called fully supported if $\supp(X,G)=X$.
From Theorem \ref{the:maximalSupportAndRMS_SM}, we observe the following. 

\begin{corollary}
\label{cor:characterisationFullySupported}
    For an action $(X,G)$ the following statements are equivalent. 
    \begin{enumerate}
        \item 
            $(X,G)$ is fully supported. 
        \item
            $\SM(X)$ is reflexive, i.e.\ $\Delta_X\subseteq \SM(X)$. 
        \item
            $\RMS(X)$ is reflexive. 
        \item 
            For some F{\o}lner sequence $\mathcal{F}$ we have that $\SM^\mathcal{F}(X)$ is reflexive. 
        \item 
            For all F{\o}lner sequences $\mathcal{F}$ we have that $\SM^\mathcal{F}(X)$ is reflexive. 
        \item 
            For some F{\o}lner sequence $\mathcal{F}$ we have that $\RMS^\mathcal{F}(X)$ is reflexive. 
        \item 
            For all F{\o}lner sequences $\mathcal{F}$ we have that $\RMS^\mathcal{F}(X)$ is reflexive. 
    \end{enumerate}
    In particular, whenever $(X,G)$ is minimal and $\mathcal{F}$ is a F{\o}lner sequence, then $\SM(X)$, $\RMS(X)$, $\SM^\mathcal{F}(X)$ and $\RMS^\mathcal{F}(X)$ are reflexive. 
\end{corollary}

\begin{remark}
    In consideration of Proposition \ref{pro:interplayweaksitmAndregionallysitm} it is natural to ask whether removing the assumption $x\neq x'$ allows for a similar characterisation for $\WSM^\mathcal{F}(X)$. 
    That this is not the case can be observed from Example \ref{exa:literatureDock}. 
\end{remark}

Furthermore, Theorem \ref{the:maximalSupportAndRMS_SM} allows to conclude the following. 

\begin{corollary}\label{cor:outsideSuppOneToOne}
    Let $(X,G)$ be an action and $\mathcal{F}$ a F{\o}lner sequence in $G$. 
    \begin{enumerate}
        \item 
        The factor map $\pi\colon X\to Y$ onto the maximal mean equicontinuous factor is \emph{one-to-one on $X\setminus \supp(X,G)$}, i.e.\ for all $x\in X\setminus \supp(X,G)$ we have $\pi^{-1}(\pi(x))=\{x\}$. 
        \item
        The factor map onto the maximal $\mathcal{F}$-mean equicontinuous factor is one-to-one on $X\setminus \supp(X,G)$.
    \end{enumerate}
\end{corollary}
\begin{proof}
    Recall from Theorem \ref{the:maximalSupportAndRMS_SM} that $\RMS(X)\subseteq \supp(X,G)^2$. 
    Since $\supp(X,G)$ is closed and invariant we observe $R:=\Delta_X\cup \supp(X,G)^2$ to be a closed invariant equivalence relation with $\icerhull{\RMS(X)}\subseteq R$.
    It follows from Theorem \ref{the:maximalMeanEquicontinuousFactor} that $R(\pi)\subseteq R$. 
    For $x\in X$ and $x'\in \pi^{-1}(\pi(x))$ we observe $(x,x')\in R(\pi)\subseteq \Delta_X\cup \supp(X,G)^2$. 
    Whenever $x\notin \supp(X,G)$ we conclude that $(x,x')\in \Delta_X$, i.e.\ $x=x'$. 
    The second statement follows from a similar argument. 
\end{proof}

\begin{remark}
    Considering Corollary \ref{cor:outsideSuppOneToOne} it is natural to ask, whether in the study of the maximal mean equicontinuous factor it is sufficient to restrict only to the maximal support. 
    This is not the case. 
    Indeed, Example \ref{exa:twoPointCompactification} below allows to observe that the maximal support of an action might be equicontinuous even if the action itself is not mean equicontinuous. 
\end{remark}

\section{Examples}
\label{sec:examples}
In this section we provide examples that illustrate and contrast our results. 

\subsection{An enlightening example}
In the following, we will revisit a transitive action of the lamplighter group discussed in \cite[Section 4]{fuhrmann2025continuity} in order to show that $\WSM^\mathcal{F}$ can fail to describe the maximal $\mathcal{F}$-mean equicontinuous factor. 
We give complete details for the convenience of the reader. 

Consider two disjoint copies $\upperElement{X}=(\mathbb{Z}\cup \{\infty\})\times \{1\}$ and $\lowerElement{X}=(\mathbb{Z}\cup \{\infty\})\times \{0\}$ of the one-point compactification of $\mathbb{Z}$. 
Let $X:=\upperElement{X} \cup \lowerElement{X}$ be the disjoint union. 
We write $\upperElement{s}:=(s,1)$ and $\lowerElement{s}:=(s,0)$ for $s\in \mathbb{Z}\cup\{\infty\}$. 	    
Let $\sigma\colon X\to X$ be the homeomorphism that fixes $\upperElement{\infty}$ and $\lowerElement{\infty}$
and which maps $\sigma(\upperElement{s}):=\upperElement{(s-1)}$
and $\sigma(\lowerElement{s}):=\lowerElement{(s-1)}$ for all $s\in \mathbb{Z}$.
Before we turn to the motivating example, we consider the action of $\mathbb{Z}$ on $X$ given by $\sigma$. 

\begin{example}
\label{exa:lamplighterZ}
    $(X,\mathbb{Z})$ is mean equicontinuous.
\end{example}
\begin{proof}
    Let $d$ be a metric that induces the topology of $X$ and denote $A_n:=\{n,\dots, 2n\}$. 
    Clearly, $\mathcal{A}=(A_n)_{n\in \mathbb{N}}$ is a F{\o}lner sequence in $\mathbb{Z}$. 
    Note that $(X,\mathbb{Z})$ decomposes into the disjoint union of the conjugated subactions $(\upperElement{X},\mathbb{Z})$ and $(\lowerElement{X},\mathbb{Z})$. 
    It thus suffices to show that $(\upperElement{X},\mathbb{Z})$ is mean equicontinuous.
    Let $U$ be an open neighbourhood of $\upperElement{\infty}$. 
    Note that for $x\in \upperElement{X}$ and sufficiently large $n$ we have that $\{g.x;\, g\in A_n\}\subseteq U$. 
    From this observation it is straightforward to show that $D_{\mathcal{A}}\equiv 0$ on 
    $(\upperElement{X})^2$. 
    In particular, we have that $(\upperElement{X},\mathbb{Z})$ is $\mathcal{A}$-mean equicontinuous. 
    Since $\mathbb{Z}$ is Abelian it follows from Proposition \ref{pro:FGL_FMEvsME} that $(\upperElement{X},\mathbb{Z})$ is mean equicontinuous. 
\end{proof}

    To revisit the example from \cite[Section 4]{fuhrmann2025continuity} we additionally consider the transposition $\tau\colon X\to X$ given by $\tau:=\left(\upperElement{0} \lowerElement{0}\right)$, i.e.\ define $\tau$ as the identity on $X\setminus \left\{\lowerElement{0},\upperElement{0}\right\}$ and such that $\tau\left(\upperElement{0}\right):=\lowerElement{0}$ and $\tau\left(\lowerElement{0}\right):=\upperElement{0}$. 
    Let $G:=\langle \sigma, \tau \rangle$ be the countable discrete group generated by $\sigma$ and $\tau$ and note that $G$ acts canonically on $X$. 
	In order to describe the group $G$ we will use the following abbreviations. 
    For $b\in \mathbb{Z}$ we denote $\tau_b:=\left(\lowerElement{b}\upperElement{b}\right)$ and observe that 
	$\tau_b=\sigma^{-b}\tau\sigma^b$ and in particular $\tau=\tau_0$. 
    Note that for $b,b'\in \mathbb{Z}$ we have $\tau_b \tau_{b'}=\tau_{b'}\tau_b$. 
    This allows to abbreviate $\tau_\mathbf{b}:=\tau_{b_1} \cdots \tau_{b_k}$ for a finite subset $\mathbf{b}\subseteq \mathbb{Z}$, where we enumerate $\mathbf{b}=\{b_1,\dots, b_k\}$. 
    As discussed in \cite[Section 4]{fuhrmann2025continuity}, $G$ is the \emph{Lamplighter group} and can be described as follows. 

\begin{lemma}\cite[Lemma 4.1]{fuhrmann2025continuity}
    Each $g\in G$ has a unique representation of the form $g=\sigma^a \tau_{\mathbf{b}}$ with $a\in \mathbb{Z}$ and $\mathbf{b}\subseteq \mathbb{Z}$ finite. 
\end{lemma}

\begin{example}
\label{exa:lamplighterExample}
    Consider the (transitive) action of $G:=\langle \sigma, \tau\rangle$ on $X$. 
    Furthermore, for $n\in \mathbb{N}$ denote 
    \[F_n:=\{\sigma^a \tau_{\mathbf{b}};\, 
    \{a\}\cup \mathbf{b}\subseteq \{n,\dots,2n\}\}.\]
    We have that
    \begin{enumerate}
        \item \label{enu:LE_Folner}
        $\mathcal{F}:=(F_n)_{n\in \mathbb{N}}$ is a F{\o}lner sequence in $G$. 
        \item \label{enu:LE_FmeanEquicontinuous}
        $(X,G)$ is $\mathcal{F}$-mean equicontinuous.
        \item \label{enu:LE_offDiagonal}
        $(\upperElement{\infty},\lowerElement{\infty})\in  \SM^\mathcal{F}(X)$. 
    \end{enumerate}
\end{example}

\begin{remark}
    Recall from Proposition \ref{pro:interplayweaksitmAndregionallysitm} that for non-diagonal pairs the notions of $\mathcal{F}$-weak sensitivity in the mean and $\mathcal{F}$-regionally joint mean sensitivity coincide. 
    Thus, (\ref{enu:LE_offDiagonal}) can be reformulated as $(\upperElement{\infty},\lowerElement{\infty})\in  \WSM^\mathcal{F}(X)$. 
\end{remark}

\begin{remark}
    From (\ref{enu:LE_offDiagonal}) and Proposition \ref{pro:basicPropertiesRMSandSM} we observe that 
    $\SM(X) \not \subseteq \Delta_X$ and hence Theorem \ref{the:theMMEFviaSM} implies that $(X,G)$ is not mean equicontinuous. 
    Thus $(X,G)$ is $\mathcal{F}$-mean equicontinuous, but not mean equicontinuous.     
    The question, whether for general actions of amenable groups the $\mathcal{F}$-mean equicontinuity for some F{\o}lner sequence $\mathcal{F}$ implies mean equicontinuity was one of the motivations for the development of this example in \cite{fuhrmann2025continuity}. 
\end{remark}

\begin{proof}[Proof of the claims in Example \ref{exa:lamplighterExample}:]
    For $n\in \mathbb{N}$ we denote $A_n:=\{n,\dots,2n\}$.
    Furthermore, we denote $\mathcal{A}:=(A_n)_{n\in \mathbb{N}}$. 

(\ref{enu:LE_Folner}):
    Consider $g=\sigma^c\tau_{\mathbf{d}}\in G$ and denote $\mathbf{d}':=\{c\}\cup \mathbf{d}$. 
    We claim that we have 
    \begin{equation}\label{eq:FolnerInclusion}
        F_n\setminus g^{-1}F_n
        \subseteq
        \bigcup_{d\in \mathbf{d}'}
        \{\sigma^a \tau_{\mathbf{b}};\, a\in A_n\setminus (A_n-d),\ \mathbf{b}\subseteq A_n\}.
    \end{equation}
    Before we prove this claim we show that it allows to deduce (\ref{enu:LE_Folner}). 
    For this note that the left multiplication with $g$ gives a bijection $G\to G$. 
    Thus (\ref{eq:FolnerInclusion}) yields 
    \begin{align*}
        \haar{(gF_n)\setminus F_n}
        =\haar{F_n\setminus g^{-1}F_n}
        \leq \sum_{d\in \mathbf{d}'}
        2^{\haar{A_n}}\,\haar{A_n\setminus (A_n-d)}. 
    \end{align*}
    Since $\mathcal{A}$ is a F{\o}lner sequence in $\mathbb{Z}$, 
    $\mathbf{d}'$ is finite and 
    $\haar{F_n}=2^{\haar{A_n}}\haar{A_n}$
    we have 
    \[
    \frac{\haar{(gF_n)\setminus F_n}}{\haar{F_n}}
    \leq
    \sum_{d\in \mathbf{d}'}
    \frac{\haar{A_n\setminus (A_n-d)}}{\haar{A_n}}
    =
    \sum_{d\in \mathbf{d}'}
    \frac{\haar{(d+A_n)\setminus A_n}}{\haar{A_n}}
    \overset{n}{\to} 0.
    \]
    This shows that $\mathcal{F}$ is a F{\o}lner sequence in $G$.

    It remains to show (\ref{eq:FolnerInclusion}). 
    For this, let 
    $\sigma^a\tau_{\mathbf{b}}\in F_n\setminus g^{-1}F_n$ and note that $a\in A_n$ and $\mathbf{b}\subseteq A_n$. 
    We need to show that there exists $d\in \mathbf{d}'$ such that $a\notin A_n-d$. 
    Since $c\in \mathbf{d}'$ this is clearly satisfied whenever $a\notin A_n-c$ and we thus assume w.l.o.g.\ that $c+a\in A_n$.     
    Furthermore, for $d\in \mathbf{d}$ we have 
    $
        \sigma^a\tau_{d+a}
        =\sigma^{a}\sigma^{-a}\sigma^{-d}\tau \sigma^d\sigma^a
        =\tau_d\sigma^a
    $ 
    and hence
    \[
        \sigma^{c+a} \tau_{(\mathbf{d}+a)\Delta \mathbf{b}}
        = \sigma^{c}\sigma^{a}\tau_{\mathbf{d}+a} \tau_{\mathbf{b}}
        =\sigma^c\tau_{\mathbf{d}}\sigma^a \tau_{\mathbf{b}}
        =g(\sigma^a \tau_{\mathbf{b}}) 
        \notin F_n.
    \]
    From $c+a\in A_n$ it follows that 
    $(\mathbf{d}+a)\Delta \mathbf{b}\not \subseteq A_n$ and $\mathbf{b}\subseteq A_n$ yields 
    $\mathbf{d}+a\not \subseteq A_n$. 
    We thus find $d\in \mathbf{d}\subseteq \mathbf{d}'$ with $a\notin A_n-d$. This proves the claim (\ref{eq:FolnerInclusion}).  

(\ref{enu:LE_FmeanEquicontinuous}):
    Let $d$ be any metric on $X$ that induces the topology.
    We denote $D_\mathcal{F}$ and $D_\mathcal{A}$ for the $\mathcal{F}$- and the $\mathcal{A}$-mean pseudometrics for the actions $(X,G)$ and $(X,\mathbb{Z})$, respectively.
    For $x\in \upperElement{X}$ there exist $s\in \mathbb{Z}\cup \{\infty\}$ such that 
    $x=\upperElement{s}$. 
    Since for large $n\in \mathbb{N}$ we have $s\notin A_n$ we observe that for such $n$ we have 
    $\sigma^a\tau_{\mathbf{b}}(x)=\sigma^a(x)$
    for all     
    $\sigma^a\tau_{\mathbf{b}}\in F_n$. 
    Clearly, a similar statement holds for every $x\in \lowerElement{X}$. 
    Thus for $x,x'\in X$ and large $n\in \mathbb{N}$ we have
    \begin{align*}
        F_n^*d(x,x')
    	&= \frac{1}{\haar{A_n}2^{\haar{A_n}}}\sum_{\mathbf{b}\subseteq A_n}\sum_{a\in A_n}{d}((\sigma^a \tau_{\mathbf{b}})(x),(\sigma^a \tau_{\mathbf{b}})(x'))\\
        &= \frac{1}{\haar{A_n}2^{\haar{A_n}}}\sum_{\mathbf{b}\subseteq A_n}\sum_{a\in A_n}{d}(\sigma^a(x),\sigma^a(x'))\\
        &= \frac{1}{\haar{A_n}}\sum_{a\in A_n}{d}(\sigma^a(x),\sigma^a(x'))=A_n^*d(x,x').
    \end{align*}
    Thus, we have $D_\mathcal{F}=D_\mathcal{A}$. 
    As presented in Example \ref{exa:lamplighterZ} above $(X,\mathbb{Z})$ is mean equicontinuous and hence 
    $D_\mathcal{F}=D_\mathcal{A}\in C(X^2)$. 
    This shows that $(X,G)$ is $\mathcal{F}$-mean equicontinuous. 

(\ref{enu:LE_offDiagonal}):
    For $n\in \mathbb{N}$ we consider $x_n:=\upperElement{n}$ and $x_n':=\upperElement{(n+1)}$. 
    Clearly, $(x_n,x_n')_{n\in \mathbb{N}}$ is an asymptotically diagonal sequence in $X$. 
    Furthermore, we have 
    \begin{align*}
        (F_n)_*\delta_{(x_n,x_n')}
        &=\frac{1}{\haar{A_n}2^{\haar{A_n}}}\sum_{\mathbf{b}\subseteq A_n}\sum_{a\in A_n}
        \delta_{((\sigma^a \tau_{\mathbf{b}})(\upperElement{n}),(\sigma^a \tau_{\mathbf{b}})(\upperElement{(n+1)}))}\\
        &=\frac{1}{4\haar{A_n}}\sum_{a\in A_n}
        \sum_{\mathbf{b}\subseteq \{n,n+1\}}\delta_{((\sigma^a \tau_{\mathbf{b}})(\upperElement{n}),(\sigma^a \tau_{\mathbf{b}})(\upperElement{(n+1)}))}\\
        &= \tfrac{1}{4} (A_n)_*(
        \delta_{(\upperElement{n},\upperElement{(n+1)})}
        +\delta_{(\upperElement{n},\lowerElement{(n+1)})}
        +\delta_{(\lowerElement{n},\upperElement{(n+1)})}
        +\delta_{(\lowerElement{n},\lowerElement{(n+1)})})\\
        &\to \tfrac{1}{4} (\delta_{(\upperElement{\infty},\upperElement{\infty})}+\delta_{(\upperElement{\infty},\lowerElement{\infty})}+\delta_{(\lowerElement{\infty},\upperElement{\infty})}+\delta_{(\lowerElement{\infty},\lowerElement{\infty})}). 
    \end{align*}
    It follows from Proposition \ref{pro:characterisationViaMeasures} that   $(\upperElement{\infty},\lowerElement{\infty})\in \SM^\mathcal{F}(X)$. 
\end{proof}

\subsection{Transitivity and dependence on the choice of a F{\o}lner sequence}
The following example will be an important building block in the construction of a more complex example. 

\begin{example}
\label{exa:twoPointCompactification}
    Let $X=\mathbb{Z}\cup \{\pm \infty\}$ be the two-point compactification of $\mathbb{Z}$.
    We consider the action $(X,\mathbb{Z})$ given by fixing $+\infty$ and $-\infty$, and $g.x:=g+x$ for $(g,x)\in \mathbb{Z}\times \mathbb{Z}$. 
    Let $\mathcal{F}=(F_k)_{k\in \mathbb{N}}$ be the standard F{\o}lner sequence in $\mathbb{Z}$ given by $F_k:=\{0,\dots,k-1\}$. 
    We have 
    \[\RMS^\mathcal{F}(X)=\SM^\mathcal{F}(X)=\RMS(X)=\SM(X)=\{\pm\infty\}^2\] 
    and hence 
    \[\RME(X)=\RME^\mathcal{F}(X)=\{\pm \infty\}^2\cup \Delta_X.\]
\end{example}
\begin{proof}
    Note that the maximal support of $(X,\mathbb{Z})$ is given by 
    \[\supp(X,\mathbb{Z})=\{\pm\infty\}.\] 
    We thus observe from Theorem \ref{the:maximalSupportAndRMS_SM} that 
    $\SM(X)\subseteq \{\pm\infty\}^2$ and that 
    \[\{(\infty,\infty),(-\infty,-\infty)\}=\Delta_{\supp(X,\mathbb{Z})}\subseteq \RMS^\mathcal{F}(X).\] 
    Furthermore, from Proposition \ref{pro:basicPropertiesRMSandSM}
    and Theorem \ref{the:RMSInSMandFolnerRealisation} 
    we know that 
    $\RMS^\mathcal{F}(X)\subseteq \SM^\mathcal{F}(X)\subseteq \SM(X)$ and that 
    $\RMS^\mathcal{F}(X)\subseteq \RMS(X)\subseteq \SM(X)$. 
    Since $\RMS^\mathcal{F}(X)$ is symmetric it remains to show that 
    $(\infty,-\infty) \in \RMS^\mathcal{F}(X)$. 
    
    For this consider the asymptotically diagonal sequence $(x_n,x_n')_{n\in \mathbb{N}}$ given by $x_n:=-n$ and $x_n':=-\infty$. 
    For all $n\in \mathbb{N}$ we have 
    \begin{align*}
        (F_k)_*\delta_{(x_n,x_n')}
        =\frac{1}{k}\sum_{i=-n}^{k-1-n}\delta_{(i,-\infty)}
        \overset{k}{\to} \delta_{(\infty,-\infty)}.  
    \end{align*}
    We thus observe from Proposition \ref{pro:characterisationViaMeasures} that $(\infty,-\infty)\in \RMS^\mathcal{F}(X)$.    
\end{proof}

\newcommand{\makeCoordinate}[3]{
\ifnum #1<0
\pgfmathsetmacro{\xx}{(-1 + 2^(#1)+1)/2}  
\fi
\ifnum #1>0
\pgfmathsetmacro{\xx}{(1 - 2^(-#1)+1)/2}  
\fi
\ifnum #1=0
\pgfmathsetmacro{\xx}{1/2}  
\fi
\ifnum #3=1
    \coordinate (P#1_#2) at (\xx+#2,0);
\fi
\ifnum #3=-1
    \coordinate (P#1_#2) at (1-\xx+#2,0);
\fi
}

\newcommand{\makeTwoPointCompactification}[3]{
\foreach \m in {-#1,...,#1} {
    \makeCoordinate{\m}{#2}{#3}
    \node[zpt] at (P\m_#2) {};
}
\ifnum #3=1
    \coordinate (P-inf_#2) at (#2,0);
    \coordinate (Pinf_#2) at (1+#2,0);
\fi
\ifnum #3=-1
    \coordinate (P-inf_#2) at (1+#2,0);
    \coordinate (Pinf_#2) at (#2,0);
\fi
\draw[->, gray] (P-inf_#2) .. controls +(-0.2,0.3) and +(0.2,0.3) .. (P-inf_#2);
\node[zptfat] at (P-inf_#2) {};
\draw[->, gray] (Pinf_#2) .. controls +(-0.2,0.3) and +(0.2,0.3) .. (Pinf_#2);
\node[zptfat] at (Pinf_#2) {};
}

\newcommand{\makeArrow}[2]{
\pgfmathtruncatemacro{\xx}{(#1)+1}
\draw[->, gray] (P#1_#2) to[bend left=45] (P\xx_#2);
}

The following example illustrates the possibility of the failure of transitivity of the considered relations and shows that also in the context of actions of $\mathbb{Z}$ we can have that $\RMS^\mathcal{F}(X)$ and $\SM^\mathcal{F}(X)$ depend on the choice of the F{\o}lner sequence $\mathcal{F}$. 

\begin{example}
\label{exa:twoPointCompactificationThreeCopies}
    Let $(X^{(i)},\mathbb{Z})$ for $i=1,2,3$ be three disjoint copies of the action considered in Example \ref{exa:twoPointCompactification}. 
    We obtain an action $(X,\mathbb{Z})$ as the quotient of $\bigcup_{i=1}^3 X^{(i)}$ after the identifications $\infty^{(1)}=\infty^{(3)}$ and $-\infty^{(2)}=-\infty^{(3)}$. 
    We simply denote $\infty^{(1)}$ and $-\infty^{(2)}$ for the respective identified elements.

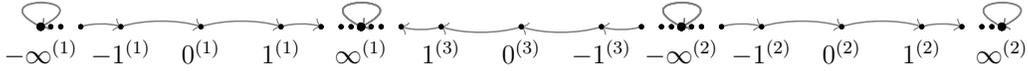
\begin{center}
    \resizebox{1\textwidth}{!}{%
\begin{tikzpicture}[xscale=3, baseline=(current bounding box.center)]
    \tikzstyle{zpt}=[circle, fill=black, inner sep=0.5pt]
    \tikzstyle{zptfat}=[circle, fill=black, inner sep=0.8pt]
    \tikzstyle{lpt}=[below=1pt, scale=0.6]
    \tikzstyle{myArrow}=[->, gray]
    \pgfmathsetmacro{\resolution}{4}
    \pgfmathsetmacro{\numberResolution}{1}

    \makeTwoPointCompactification{\resolution}{1}{1}
    \makeTwoPointCompactification{\resolution}{2}{-1}
    \makeTwoPointCompactification{\resolution}{3}{1}

    \foreach \m in {-2,...,1} {
    \makeArrow{\m}{1}
    \makeArrow{\m}{2}
    \makeArrow{\m}{3}
    }

    \makeTwoPointCompactification{\resolution}{1}{1}
    \makeTwoPointCompactification{\resolution}{2}{-1}
    \makeTwoPointCompactification{\resolution}{3}{1}

    \node[lpt] at (P-1_1) {$-1^{(1)}$};
    \node[lpt] at (P-1_2) {$-1^{(3)}$};
    \node[lpt] at (P-1_3) {$-1^{(2)}$};
    \node[lpt] at (P0_1) {$0^{(1)}$};
    \node[lpt] at (P0_2) {$0^{(3)}$};
    \node[lpt] at (P0_3) {$0^{(2)}$};
    \node[lpt] at (P1_1) {$1^{(1)}$};
    \node[lpt] at (P1_2) {$1^{(3)}$};
    \node[lpt] at (P1_3) {$1^{(2)}$};

    \node[lpt] at (P-inf_1) {$-\infty^{(1)}$};
    \node[lpt] at (Pinf_1) {$\infty^{(1)}$};
    \node[lpt] at (P-inf_3) {$-\infty^{(2)}$};
    \node[lpt] at (Pinf_3) {$\infty^{(2)}$};
\end{tikzpicture}%
}
\end{center}
    As above, let $\mathcal{F}=(F_k)_{k\in \mathbb{N}}$ be the standard F{\o}lner sequence in $\mathbb{Z}$ given by $F_k:=\{0,\dots,k-1\}$.
    Furthermore, let $\mathcal{F}'$ be the F{\o}lner sequence given by $F_k':=\{-k,\dots, k\}$.
    Note that we have $\supp(X,\mathbb{Z})=\{\pm \infty^{(1)}, \pm\infty^{(2)}\}$. 
    Denoting 
    \[R:=\supp(X,\mathbb{Z})^2=\{\pm \infty^{(1)}, \pm\infty^{(2)}\}^2\]
    we have the following:
    \begin{enumerate}
        \item \label{enu:ExampleMeanEquicontinuousStructureRelation}
        $\RME(X)=\RME^\mathcal{F}(X)=\RME^{\mathcal{F}'}(X)=R\cup \Delta_X$. 
        \item \label{enu:ExampleRMS}
        $\RMS(X)=\SM(X)=\RMS^{\mathcal{F}'}(X)=\SM^{\mathcal{F}'}(X)$.
        \item \label{enu:ExampleRMS_II}
        $\RMS(X)=R\setminus \{(-\infty^{(1)},\infty^{(2)}), (\infty^{(2)},-\infty^{(1)})\}.$
        \item \label{enu:ExampleRMS_F}
        $\RMS^\mathcal{F}(X)
        =\SM^\mathcal{F}(X)
        =\RMS(X)\setminus \{(-\infty^{(1)},-\infty^{(2)}), (-\infty^{(2)},-\infty^{(1)})\}.$
    \end{enumerate}
    In particular, we observe that $\RMS(X)$, $\RMS^\mathcal{F}(X)$, $\SM(X)$ and $\SM^\mathcal{F}(X)$ are not transitive. 
    Furthermore, we observe that we have
    $\RMS^\mathcal{F}(X)\subsetneq \RMS(X)=\RMS^{\mathcal{F}'}(X)$ and $\SM^\mathcal{F}(X)\subsetneq \SM(X)=\SM^{\mathcal{F}'}(X)$.  
\end{example}
\begin{proof}    
(\ref{enu:ExampleMeanEquicontinuousStructureRelation}):
    It follows from Theorem \ref{the:maximalSupportAndRMS_SM} that $\RMS(X)\subseteq \supp(X,\mathbb{Z})^2=R$. 
    Since $R\cup \Delta_X$ is a closed invariant equivalence relation on $X$ we thus have 
    \[\RME(X)=\icerhull{\RMS(X)}\subseteq R\cup \Delta_X.\] 
    For the converse, by a similar argument as in Example \ref{exa:twoPointCompactification}, one has 
    \[\{(-\infty^{(1)},\infty^{(1)}), (\infty^{(1)},-\infty^{(2)}), (-\infty^{(2)},\infty^{(2)})\}\subseteq \RMS^\mathcal{F}(X)\subseteq \RME(X).\]
    Since $\RME(X)$ is transitive we thus have 
    \[R=\{\infty^{(1)}, -\infty^{(1)}, \infty^{(2)}, -\infty^{(2)}\}^2\subseteq\RME(X)\]
    and it follows from the reflexivity of $\RME(X)$ that 
    $R\cup \Delta_X\subseteq \RME(X).$
    This shows $\RME(X)=R\cup \Delta_X$. 
    A similar argument allows to observe 
    \[\RME^\mathcal{F}(X)=\RME^{\mathcal{F}'}(X)=R\cup \Delta_X.\]  

(\ref{enu:ExampleRMS}, \ref{enu:ExampleRMS_II} \& \ref{enu:ExampleRMS_F}):
    Recall from Proposition \ref{pro:basicPropertiesRMSandSM}
    and Theorem \ref{the:RMSInSMandFolnerRealisation} 
    that 
    $\RMS^\mathcal{F}(X)\subseteq \SM^\mathcal{F}(X)\subseteq \SM(X)$ and that 
    $\RMS^\mathcal{F}(X)\subseteq \RMS(X)\subseteq \SM(X)$.
    A similar statement also holds for $\RMS^{\mathcal{F}'}(X)$ and $\SM^{\mathcal{F}'}(X)$. 
    Furthermore, as above, we observe 
    \begin{align*}
        \{(-\infty^{(1)},\infty^{(1)}), (\infty^{(1)},-\infty^{(2)}), (-\infty^{(2)},\infty^{(2)})\}\subseteq \RMS^\mathcal{F}(X)\subseteq \SM(X)\subseteq R. 
    \end{align*}
    Thus, it suffices to show that 
    \begin{itemize}
        \item $(\infty^{(1)},\infty^{(2)})\in \RMS^\mathcal{F}(X)$,
        \item $(-\infty^{(1)},-\infty^{(2)})\in \RMS^{\mathcal{F}'}(X)$,
        \item $(-\infty^{(1)},\infty^{(2)})\notin \SM(X)$, and
        \item $(-\infty^{(1)},-\infty^{(2)})\notin \SM^\mathcal{F}(X)$. 
    \end{itemize}
    Clearly, we have $(-n^{(3)},-n^{(2)})\to (-\infty^{(2)},-\infty^{(2)})\in \Delta_X$. 
    Furthermore, we have
    \[(F_k)_*\delta_{(-n^{(3)},-n^{(2)})}
    =\frac{1}{k}\sum_{i=-n}^{k-1-n}\delta_{(i^{(3)},i^{(2)})}
    \overset{k}{\to} \delta_{(\infty^{(1)},\infty^{(2)})}\] 
    for all $n\in \mathbb{N}$. 
    Thus, Proposition \ref{pro:characterisationViaMeasures}
    yields
    $(\infty^{(1)},\infty^{(2)})\in \RMS^\mathcal{F}(X)$. 
    Similarly, 
    \[(F_k')_*\delta_{(n^{(1)},n^{(3)})}
    =\frac{1}{2k+1}\sum_{i=n-k}^{n+k}\delta_{(i^{(1)},i^{(3)})}
    \overset{k}{\to} \tfrac{1}{2}(\delta_{(\infty^{(1)},\infty^{(1)})}
    +\delta_{(-\infty^{(1)},-\infty^{(2)})})\] 
    for all $n\in \mathbb{N}$
    yields 
    $(-\infty^{(1)},-\infty^{(2)})\in \RMS^{\mathcal{F}'}(X)$. 
      
    A straightforward argument shows that $(-\infty^{(1)},\infty^{(2)})$ is not regionally proximal. We thus observe from Proposition \ref{pro:basicPropertiesRMSandSM} that 
    $(-\infty^{(1)},\infty^{(2)})\notin \SM(X)$. 
    
    To show that $(-\infty^{(1)},-\infty^{(2)})\notin \SM^\mathcal{F}(X)$ consider the open neighbourhoods
    $U^{(1)}:=\{-\infty^{(1)}\}\cup \{-m^{(1)};\, m\in \mathbb{N}\}$ and 
    \[U^{(2)}:=\{-\infty^{(2)}\}\cup \{-m^{(2)},-m^{(3)};\, m\in \mathbb{N}\}\]
    of $-\infty^{(1)}$ and $-\infty^{(2)}$, respectively. 
    Denote $U:=U^{(1)}\times U^{(2)}$. 
Assume for a contradiction that $\bigl(-\infty^{(1)},-\infty^{(2)}\bigr)\in \SM^{\mathcal F}(X)$.
Then, arguing as in the proof of Proposition~\ref{pro:infimumReformulation}, we
can find an asymptotically diagonal sequence $(n_k)_{k\in\mathbb N}$ such that
\[\inf_{k\in \mathbb{N}}\frac{\haar{G_U(x_{n_k},x_{n_k}')\cap F_{n_k}}}{\haar{F_{n_k}}}>0.\]
    For $k\in \mathbb{N}$ there exists $g\in F_{n_k}$ with 
    $(g.x_{n_k},g.x_{n_k}')\in U$. 
    It follows from $g\geq 0$ and the choice of $U$ that we then have $(x_{n_k},x_{n_k}')\in U$, contradicting the asymptotic diagonality of $(x_n,x_n')_{n\in \mathbb{N}}$. 
    This shows $(-\infty^{(1)},-\infty^{(2)})\notin \SM^\mathcal{F}(X)$. 
\end{proof}

\bibliographystyle{alpha}
\bibliography{ref}

\end{document}